\documentclass[10pt]{article}
\usepackage[a4paper]{geometry}
\usepackage{amsthm}
\usepackage{amsmath}
\usepackage{amssymb}
\usepackage{epsfig,graphicx}
\usepackage{subfigure}

\begin{document}

\newtheorem{lemma}{Lemma}[section]
\newtheorem{theorem}{Theorem}[section]
\newtheorem{Condition}{Condition}[section]
\newtheorem{Integrator}{Integrator}
\newtheorem{Remark}{Remark}[section]

\title{From efficient symplectic exponentiation of matrices to
symplectic integration of high-dimensional Hamiltonian systems with slowly varying quadratic stiff potentials}

\author{Molei Tao\footnotemark[2] \footnotemark[4], Houman Owhadi\footnotemark[2] \footnotemark[3], Jerrold E. Marsden\footnotemark[2] \footnotemark[3]}

\renewcommand{\thefootnote}{\fnsymbol{footnote}}
\footnotetext[2]{Control \& Dynamical Systems, MC 107-81,}
\footnotetext[3]{Applied \& Computational Mathematics, MC 217-50, California Institute of Technology, Pasadena, CA 91125, USA}
\footnotetext[4]{Corresponding author; Email: mtao@caltech.edu}
\renewcommand{\thefootnote}{\arabic{footnote}}

\maketitle

\begin{abstract}
 We present a multiscale integrator for Hamiltonian systems with slowly varying quadratic stiff potentials that uses coarse timesteps (analogous to what the impulse method uses for constant quadratic stiff potentials). This method is based on the highly-non-trivial introduction of two efficient symplectic schemes for exponentiations of matrices that only require $\mathcal{O}(n)$ matrix multiplications operations at each coarse time step for a preset small number $n$. The proposed integrator is shown to be
  (i) uniformly convergent on positions; (ii)  symplectic in both slow and fast variables; (iii) well adapted to high dimensional systems.
  Our framework also provides a general method for iteratively exponentiating a slowly varying sequence of (possibly high dimensional) matrices in an efficient way.
\end{abstract}

\maketitle

\section{Introduction}

One objective of this paper is to obtain an explicit and efficient numerical integration algorithm for the following multiscale Hamiltonian system:
\begin{equation}
    \left\{ \begin {array} {rcl}
    M \begin{bmatrix} \dot{q}^{fast} \\ \dot{q}^{slow} \end{bmatrix} &=& \begin{bmatrix} p^{fast} \\ p^{slow} \end{bmatrix} \\
    \begin{bmatrix} \dot{p}^{fast} \\ \dot{p}^{slow} \end{bmatrix} &=& -\nabla V(q^{fast},q^{slow})-\epsilon^{-1}\nabla U(q^{fast},q^{slow})
    \end {array} \right.
    \label{ODEs}
\end{equation}
where $q^{slow},p^{slow}$  and $q^{fast},p^{fast}$ are slow and fast degrees of freedom (in the sense that slow degrees of freedom have bounded time derivatives, whereas time derivatives of fast ones may grow unboundedly as $\epsilon \rightarrow 0$). Observe that a direct numerical integration of \eqref{ODEs} becomes prohibitive as $\epsilon \downarrow 0$. Notice also that not all stiff Hamiltonian systems are multiscale, and whether a separation of timescales exists depends on specific forms of $V(\cdot)$, $U(\cdot)$ and initial conditions. To the authors' knowledge, a generic theory that determines whether a stiff system is multiscale has not been fully developed yet.

We will mainly discuss and analyze the case where $U(q^{fast},q^{slow})=\frac{1}{2}[q^{fast}]^T K(q^{slow})q^{fast}$, which we call a quasi-quadratic potential throughout this paper. In this case, the proposed method will be able to integrate the system using a coarse timestep. Notice that if $K$ remains constant with  respect to $q^{slow}$, then the impulse method \cite{Grubmuller:91,Tuckerman:92,Skeel:99, SIM1} allows for an accurate and symplectic (see for instance \cite{Hairer:04} for a definition) integration of
\eqref{ODEs} using coarse steps. The impulse method can, in principle, integrate the situation where $K$ is a regular function of slow variables; however, its practical implementation requires a numerical approximation to the stiff system
\begin{equation}
\begin{cases}
\ddot{q}^{fast}&=-\epsilon^{-1}\partial U/\partial q^{fast} (q^{fast},q^{slow}) \\
\ddot{q}^{slow}&=-\epsilon^{-1}\partial U/\partial q^{slow} (q^{fast},q^{slow})
\end{cases},
\end{equation}
which generally needs to be based on a numerical integration with small steps. The advantage of the impulse method over Verlet is that $\nabla V$ only needs to be evaluated at coarse timesteps, but nevertheless its computational cost blows up as $\epsilon\rightarrow 0$.

To use a coarse integration timestep independent of $\epsilon$, we adopt a splitting approach to treat the slow and fast variables separately. At each coarse step, we will require an exact solution or a numerical approximation to the following stiff system:
\begin{equation}
\begin{cases}
\ddot{q}^{fast}&=-\epsilon^{-1}\partial U/\partial q^{fast} (q^{fast},\cdot) \\
\dot{p}^{slow}&=-\epsilon^{-1}\partial U/\partial q^{slow} (q^{fast},\cdot)
\end{cases},
\label{fastonly}
\end{equation}
in which $q^{slow}$ is fixed (different from the impulse method). To obtain such an approximation, we compute the exponential of a matrix that depends on $K(q^{slow})$ and $\partial K(q^{slow})$. Still, if not handled appropriately, the cost of this method blows up rapidly as $\epsilon$ decreases and/or the dimension of the system increases. Furthermore, symplecticity would also be jeopardized by inaccuracies of the numerical exponentiations.

In this paper, we propose an integrator well-adapted to high-dimensional systems, which computes the exponentiation in an efficient and symplectic way. Only $\mathcal{O}(n)$ matrix multiplication operations at each coarse time step are needed, where $n$ is a preset small integer at most $\log \epsilon^{-1}$.
Although simple in appearance,  to guarantee the symplecticity (in all variables) of the resulting method without  diagonalizing  $K(q^{slow})$ (which is expensive) is a surprisingly difficult problem, and in fact it is highly non-trivial even when $K(q^{slow})$ is a scalar \cite{LeBris:07}.

In addition to a solution to this problem, this paper also provides a general method for iteratively exponentiating a slowly varying sequence of (possibly high dimensional) matrices  in an efficient way (see Section \ref{bliblibloublou} and the Appendix). This method works for any matrices, and it is not restricted to the integration of \eqref{ODEs}. The preservation of symplecticity associated with these two proposed matrix exponentiation schemes (both suit high dimensional systems; the first one is in Section \ref{gguqirehugqubliafd}) is a core difficulty addressed in this paper.

Also, useful discrete geometric structures regarding the flow map of a parameter-dependent vectorial harmonic oscillator have been studied; for instance, derivatives of its flow map with respect to the parameter can be computed using matrix exponentials.

Although backward error analysis (relating symplecticity and energy conservation) does not apply directly to stiff systems (due to large Lipschitz constants),  we numerically observe improved long time behaviors for the proposed integrator, such as near-preservation of energy and conservation of momentum maps. We also note that modulated Fourier expansion \cite{CoHaLu03} has been proposed to explain favorable long time energy behaviors of some integrators for oscillatory Hamiltonian systems.

We prove the uniform (in $\epsilon^{-1}$) convergence of the method and bound the global error
on  $q^{slow}$ by  $C H$, where $H$ is a coarse integration timestep and $C$ is a constant independent of $\epsilon^{-1}$.

\section{The proposed method}
\subsection{General methodology}
\label{SectionGIM}
Consider a Hamiltonian system with the following Hamiltonian:
\begin{equation}
    \mathcal{H}(q,p)=\frac{1}{2}p^T M^{-1}p+V(q)+\epsilon^{-1}U(q^{fast},q^{slow}),
    \label{quasi-quadratic}
\end{equation}
where $0 < \epsilon\ll 1$, $q \in \mathbb{R}^d$, $p \in \mathbb{R}^d$. $q^{fast} \in \mathbb{R}^{d_f}$ and $q^{slow} \in \mathbb{R}^{d_s}$ are fast and slow variables. For the sake of clarity, we will assume that $q=(q^{fast},q^{slow})$ (and hence $d=d_f+d_s$), but the method presented here can be generalized to the situation where $(q^{fast},q^{slow})=\eta(q)$, where $\eta$ is a diffeomorphism explicitly known beforehand (for instance, see \cite{DoLeLe10} for an example where a diffeomorphism puts an extensible pendulum into a quasi-quadratic form). The idea of separating the variables is that fast variables need to be integrated using  $o(\sqrt{\epsilon})$ timesteps by a single scale integrator, whereas slow variables can be resolved using  $o(1)$ steps. Dividing variables into different timescales is a widely used technique in studying stiff and multiscale systems (e.g., \cite{MR2314852, KevGio09}). A rigorous definition of separation of timescales can be found in \cite{FLAVOR09}, for instance.

Without loss of generality, we can further assume $M$ to be the identity matrix.
The governing ODE system is
\begin{equation}
    \left\{ \begin{array}{rcl}
    \dot{q}^{fast} &=& p^{fast} \\
    \dot{p}^{fast} &=& -\epsilon^{-1}\frac{\partial U}{\partial q^{fast}}-\frac{\partial V}{\partial q^{fast}} \\
    \dot{q}^{slow} &=& p^{slow} \\
    \dot{p}^{slow} &=& -\epsilon^{-1}\frac{\partial U}{\partial q^{slow}}-\frac{\partial V}{\partial q^{slow}}
    \end{array} \right.
    \label{fullSystem}
\end{equation}
which can be split into a sum of three vector fields:
\[
    \left\{ \begin{array}{rcl}
    \dot{q}^{fast} &=& 0 \\
    \dot{p}^{fast} &=& 0 \\
    \dot{q}^{slow} &=& p^{slow} \\
    \dot{p}^{slow} &=& 0
    \end{array} \right. \\
    \left\{ \begin{array}{rcl}
    \dot{q}^{fast} &=& 0 \\
    \dot{p}^{fast} &=& -\frac{\partial V}{\partial q^{fast}} \\
    \dot{q}^{slow} &=& 0 \\
    \dot{p}^{slow} &=& -\frac{\partial V}{\partial q^{slow}}
    \end{array} \right.
    \left\{ \begin{array}{rcl}
    \dot{q}^{fast} &=& p^{fast} \\
    \dot{p}^{fast} &=& -\epsilon^{-1}\frac{\partial U}{\partial q^{fast}} \\
    \dot{q}^{slow} &=& 0 \\
    \dot{p}^{slow} &=& -\epsilon^{-1}\frac{\partial U}{\partial q^{slow}}
    \end{array} \right.
\]
such that the exact flow of each could be obtained, which is also symplectic at the same time. Indeed, denote the flow maps of all systems by $\phi^i(s),i=1,2,3$ over a time of $s$. It is easy to see that they are all symplectic.

Observe that $\phi^1$ and $\phi^2$ are analytically available. We only consider the case where $\phi^3$ is also analytically or numerically known; more precisely, the numerical solution $\tilde{\phi}^3$ has to have a consistent uniform local error over a coarse time step $H=o(1)$, i.e., $\| \tilde{\phi}^3(H)-\phi^3(H) \| \leq CH^2$ for a constant $C$ independent of $\epsilon^{-1}$. This can be satisfied for arbitrary $U(\cdot)$ by a symplectic integration with a microscopic timestep $h=o(\sqrt{\epsilon})$, which is in the same spirit as the impulse method. One step update of the proposed method is obtained by composing the three flow maps: $\phi^1(H) \circ \phi^2(H) \circ \phi^3(H)$. Notice that any split can result in a convergent numerical scheme, but this particular split treats two timescales independently and therefore is uniformly convergent at least in the quasi-quadratic stiff potential case (illustrated later); also, it results in a symplectic scheme.

\begin{Remark}
    If there were no slow variable, we would compose the flows of $\begin{cases}
    \dot{q}^{fast} = p^{fast} \\
    \dot{p}^{fast} = -\epsilon^{-1}\frac{\partial U}{\partial q^{fast}}
    \end{cases}$ and $\begin{cases}
    \dot{q}^{fast} = 0 \\
    \dot{p}^{fast} = -\frac{\partial V}{\partial q^{fast}}
    \end{cases}$ and obtain a first-order version of the original impulse method.
    \label{degenerateToImpulse}
\end{Remark}

\begin{Remark}
    There are also alternative higher-order ways of composing these flow maps; see, for instance, \cite{Hairer:04, Neri:88}. In fact, the original impulse method is second-order and can be constructed from a second-order composition scheme. However, we will stick to first-order Lie-Trotter ($\phi^1(H) \circ \phi^2(H) \circ \phi^3(H)$) in this paper.
\end{Remark}

\subsection{Quasi-quadratic fast potentials}
We will, from now on, discuss and analyze an analytically solvable case (so that a uniform coarse timestep could be used), in which $U=\frac{1}{2}{[q^{fast}]}^T K(q^{slow}) q^{fast}$, where $K$ is a positive definite $d_f$-by-$d_f$ matrix valued function. This fast potential represents stiff harmonic oscillators with non-constant but slowly varying frequencies. We call such potentials quasi-quadratic. In this case, the first split vector field is
\begin{equation}\label{eqqasiquadflow}
    \left\{ \begin{array}{rcl}
    \dot{q}^{fast} &=& p^{fast} \\
    \dot{p}^{fast} &=& -\epsilon^{-1}K(q^{slow})q^{fast} \\
    \dot{q}^{slow} &=& 0 \\
    \dot{p}^{slow} &=& -\epsilon^{-1}\frac{1}{2} {[q^{fast}]}^T \nabla K(q^{slow}) q^{fast}
    \end{array} \right.
\end{equation}
where the last equation is understood as $\dot{p}^{slow}_i = -\epsilon^{-1}\frac{1}{2} {[q^{fast}]}^T \partial_i K(q^{slow}) q^{fast}$ for $i=1,\ldots,d_s$.
The flow of this dynamical system on $q^{fast}$ and $p^{fast}$ is just an exponential map, which in this case corresponds to linear combinations of initial conditions with trigonometric coefficients. For $p^{slow}$, because $q^{slow}$ (and hence $\nabla K(q^{slow})$) is fixed, one could obtain its exact flow by analytically integrating a quadratic function of trigonometric functions.

When $d_f=1$, the exact flow map of \eqref{eqqasiquadflow} over time $H$ is (letting $\omega=\sqrt{\epsilon^{-1}K(q^{slow})}$):
\begin{equation}
\begin{cases}
    q^{fast} &\mapsto \cos(\omega H) q^{fast}+\sin(\omega H)/\omega p^{fast} \\
    p^{fast} &\mapsto -\omega \sin(\omega H) q^{fast}+\cos(\omega H) p^{fast} \\
    q^{slow} &\mapsto q^{slow} \\
    p^{slow} &\mapsto p^{slow}-\epsilon^{-1}\frac{1}{2} \nabla K(q^{slow}) \frac{1}{4\omega^3} \big(2\omega(H[p^{fast}]^2+p^{fast}q^{fast}+\omega^2 H[q^{fast}]^2) \\ & \qquad\qquad\qquad -2\omega p^{fast}q^{fast} \cos(2\omega H)+(-[p^{fast}]^2+\omega^2 [q^{fast}]^2) \sin(2 \omega H)\big)
\end{cases}
\end{equation}
where again the last equation is understood as
\begin{eqnarray}
    p^{slow}_i &\mapsto p^{slow}_i-\epsilon^{-1}\frac{1}{2} \partial_i K(q^{slow}) \frac{1}{4\omega^3} \big(2\omega(H[p^{fast}]^2+p^{fast}q^{fast}+\omega^2 H[q^{fast}]^2) \nonumber\\ & \qquad\qquad\qquad -2\omega p^{fast}q^{fast} \cos(2\omega H)+(-[p^{fast}]^2+\omega^2 [q^{fast}]^2) \sin(2 \omega H)\big)
\end{eqnarray}

When $d_f\geq 2$, the obvious method to obtain the exact flow of \eqref{eqqasiquadflow} is based on
a diagonalization of $K$.
 More precisely, since $K$ is symmetric, we can write  $\epsilon^{-1}K(q^{slow})=\epsilon^{-1}Q(q^{slow})^T D(q^{slow}) Q(q^{slow})$, where $\epsilon^{-1}D(q^{slow})=\text{diag}[\omega_1^2,\ldots,\omega_{d_f}^2]$). Then
 \begin{align}
    &\exp \left( \begin{bmatrix} 0 & HI \\ -\epsilon^{-1}HK(q^{slow}) & 0 \end{bmatrix} \right) = \begin{bmatrix} Q^T & 0 \\ 0 & Q^T \end{bmatrix} \exp \left( \begin{bmatrix} 0 & HI \\ -\epsilon^{-1}HD & 0 \end{bmatrix} \right) \begin{bmatrix} Q & 0 \\ 0 & Q \end{bmatrix}
    = \begin{bmatrix} Q^T & 0 \\ 0 & Q^T \end{bmatrix} \cdot \nonumber \\
    & \begin{bmatrix} \text{diag}[\cos (\omega_1 H), \ldots,\cos(\omega_{d_f} H)] & \text{diag}[\sin (\omega_1 H)/\omega_1, \ldots,\sin (\omega_{d_f} H)/\omega_{d_f}] \\ \text{diag}[-\sin (\omega_1 H)\omega_1, \ldots,-\sin (\omega_{d_f} H)\omega_{d_f}] & \text{diag}[\cos (\omega_1 H), \ldots,\cos(\omega_{d_f} H)] \end{bmatrix} \begin{bmatrix} Q & 0 \\ 0 & Q \end{bmatrix}
    \label{expViaDiagonalization}
\end{align}
A similar (but lengthy) calculation will give the expression of the flow on $p^{slow}$.

If the diagonalization frame of $K(\cdot)$ is constant, i.e., $Q$ does not depend on $q^{slow}$, then $Q$ needs to be computed only once throughout the simulation, and then the calculation of the flow on $q^{fast}$ and $p^{fast}$ is dominated by the cost of 2 matrix multiplication operations per coarse step (at expense of $\mathcal{O}({d_f}^{2.376})$ per multiplication by the state-of-art Coppersmith-Winograd algorithm \cite{CoWi90}). However, if the frame varies
  ($Q$ depends on $q^{slow}$), then diagonalizing $K$ at each time step can offset the gain obtained by the macro-time-stepping of the algorithm. This is especially true if $d_f$ is large. Moreover, errors in numerical diagonalizations may accumulate and deteriorate the symplecticity of $\phi^3$.

In this paper, we address those difficulties by proposing a method, described below, for  the numerical integration of \eqref{eqqasiquadflow} that is symplectic and that remains computationally tractable in high-dimensional cases (large $d_f$).

\subsection{Fast numerical matrix exponentiation for the symplectic integration of \eqref{eqqasiquadflow}}
\label{gguqirehugqubliafd}
The proposed method is based on matrix exponentiation. We will first describe its analytical formulation, and then present an accurate numerical approximation that is both symplectic and computationally cheap.

The first step of our method is based on the following property of matrix exponentials illustrated in \cite{Va78}: if $N$ and $M$ are constant square matrices of the same dimension, then
\begin{equation}
 \exp\left(\begin{bmatrix} -N^T & M \\ 0 & N \end{bmatrix} H\right) = \begin{bmatrix} F_2(H) & G_2(H) \\ 0 & F_3(H) \end{bmatrix}
\label{pureMatrixExponentiation}
\end{equation}
with
\begin{equation}\label{ghfdgd454d54d4}
\begin{cases}
    F_2(H) &= \exp(-N^T H) \\
    F_3(H) &= \exp(N H) \\
    F_3(H)^T G_2(H) &= \int_0^H \exp (N^T s) M \exp (N s) \,ds
\end{cases}
\end{equation}
Therefore,  ordering coordinates as $q^{fast},p^{fast}$,  taking $N:=\begin{bmatrix} 0 & I \\ -\epsilon^{-1}K(q^{slow}) & 0 \end{bmatrix}$ and $M_i:=\begin{bmatrix} \epsilon^{-1} \partial_i K (q^{slow}) & 0 \\ 0 & 0 \end{bmatrix}$ with $i=1,\ldots,d_s$ which indicates the component of the slow variable, we obtain that
if
\begin{equation}
 \begin{bmatrix} F_2(H) & G_{2,i}(H) \\ 0 & F_3(H) \end{bmatrix} := \exp\left(\begin{bmatrix} -N^T & M_i \\ 0 & N \end{bmatrix} H\right)
\label{pureMatrixExponentiationii}
\end{equation}
then the (linear) flow map on $q^{fast},p^{fast}$ is given by
\begin{equation}
    \exp(N H) = F_3(H)
\end{equation}
and the drift on $p^{slow}$ is given by
\begin{align}\label{ghfdsed54d4}
    & \int_t^{t+H} q^{fast}(s)^T \epsilon^{-1} \partial_i K (q^{slow}) q^{fast}(s) \,ds = \int_0^H \begin{bmatrix} q^{fast}(t) \\ p^{fast}(t) \end{bmatrix}^T \exp (N^T s) M_i \exp (N s) \begin{bmatrix} q^{fast}(t) \\ p^{fast}(t) \end{bmatrix} \,ds \nonumber\\
    & \qquad \qquad =
    \begin{bmatrix} q^{fast}(t) \\ p^{fast}(t) \end{bmatrix}^T F_3(H)^T G_{2,i}(H) \begin{bmatrix} q^{fast}(t) \\ p^{fast}(t) \end{bmatrix}
\end{align}

Therefore,  $\phi^3(H)$ is given by:
\begin{equation}
\begin{cases}
\begin{bmatrix} q^{fast} \\ p^{fast} \end{bmatrix} &\mapsto F_3(H) \begin{bmatrix} q^{fast} \\ p^{fast} \end{bmatrix}
\\
q^{slow} &\mapsto q^{slow} \\
p^{slow}_i &\mapsto p^{slow}_i-\frac{1}{2}  \begin{bmatrix} q^{fast} \\ p^{fast} \end{bmatrix}^T F_3(H)^T G_{2,i} (H) \begin{bmatrix} q^{fast} \\ p^{fast} \end{bmatrix}
\end{cases}
\label{phi3}
\end{equation}
where in the last equation $i=1,\ldots,d_s$.

In addition, our specific choice of $M_i$ is a symmetric matrix for each $i$, because $K(\cdot)$ is symmetric. Consequently, $\exp (N^T s) M_i \exp (N s)$ is symmetric, and therefore
\begin{equation}
    F_3(H)^T G_{2,i}(H) = (F_3(H)^T G_{2,i}(H))^T
    \label{symmetricMatrix}
\end{equation}

Assuming we have $F_3$ and $G_{2,i}$ (which will be given by Integrator \ref{fastMatrixExpScheme}), \eqref{quasi-quadratic} can be integrated by the following:

\begin{Integrator}
Symplectic multi-scale integrator for \eqref{quasi-quadratic} with $U=\frac{1}{2}{[q^{fast}]}^T K(q^{slow}) q^{fast}$. Its one-step update mapping $q_{k},p_{k}$ onto $q_{k+1},p_{k+1}$ with a coarse timestep $H$ is given by:
\begin{align}
    &\left\{ \begin{array}{rcl}
        q_{k'}^{slow} &=& q_{k}^{slow}+H p_{k}^{slow} \\
        q_{k'}^{fast} &=& q_{k}^{fast} \\
        p_{k'}^{slow} &=& p_{k}^{slow}-H\partial V/ \partial q^{slow} (q_{k'}^{slow},q_{k'}^{fast}) \\
        p_{k'}^{fast} &=& p_{k}^{fast}-H\partial V/ \partial q^{fast} (q_{k'}^{slow},q_{k'}^{fast}) \\
    \end{array} \right. \label{phi12} \\
    &\left\{ \begin{array}{rcl}
        \begin{bmatrix}
        q_{k+1}^{fast} \\
        p_{k+1}^{fast}
        \end{bmatrix}
        &=& F_{3,k}
        \begin{bmatrix}
        q_{k'}^{fast} \\
        p_{k'}^{fast}
        \end{bmatrix} \\
        q_{k+1}^{slow} &=& q_{k'}^{slow} \\
        p_{k+1,i}^{slow} &=& p_{k',i}^{slow} - \frac{1}{2} \begin{bmatrix}
        q_{k'}^{fast} \\
        p_{k'}^{fast}
        \end{bmatrix}^T F_{3,k}^T G_{2,k,i} \begin{bmatrix}
        q_{k'}^{fast} \\
        p_{k'}^{fast}
        \end{bmatrix}
    \end{array} \right. \label{phi3numerical}
\end{align}
where $F_{2,k}$, $G_{2,k,i}$ ($i=1,\ldots,d_s$) and $F_{3,k}$ are numerical approximations of that in \eqref{pureMatrixExponentiationii} at each time step $k'$ (using $q^{slow}_{k'}$), for instance computed by Integrator \ref{fastMatrixExpScheme}.
\label{extended_SIM}
\end{Integrator}

To numerically approximate the above flow map \eqref{phi3}, i.e., to obtain $F_{3,k}$ and $G_{2,k,i}$, we need to ensure two points: (i) an approximation of the matrix exponential (and hence $F_{3,k}$ and $G_{2,k,i}$) will not affect the symplecticity of the resulting approximation of $\phi^3$; (ii) the numerical computation of the exponential \eqref{pureMatrixExponentiationii} will not off-set the savings gained by using a coarse timestep. It is highly non-trivial to satisfy both simultaneously, because most matrix exponentiation methods will ruin the symplecticity of $\phi^3$ unless high precision (much higher than the requirement on accuracy) is enforced, but then the computational cost will be high. In fact, a necessary and sufficient condition for symplecticity is given by Lemma \ref{lemma_flowSymplecticity}, and it is unclear how most matrix exponentiation methods, for instance those based on matrix decompositions (e.g., diagonalization, QR decomposition) with computational costs of $Cd_f^3$ flops, will satisfy this condition (unless $C$ is very large and the approximation is very accurate). Also, here is an illustration of a popular non-decomposition-based exponentiation method that fails to satisfy this symplecticity condition:
\paragraph{Example:}
MATLAB function `expm' \cite{ScalingSquaring} uses a scaling and squaring strategy based on the following identity:
\begin{equation}
    \exp(X)=[\exp (X/2^n)]^{2^n}
\end{equation}
where $n$ is a big enough preset integer such that $X/2^n$ has a small norm, and therefore Pad\'{e} approximation \cite{ScalingSquaring} could be employed to approximate $\exp (X/2^n)$. The simplest (1,0) Pad\'{e} approximation, which is essentially Taylor expansion to 1st-order, gives
\begin{equation}
    \exp(X) \approx [I+X/2^n]^{2^n}
    \label{expm}
\end{equation}
However, this approximation is not symplectic. For instance, consider a counterexample of $X=\begin{bmatrix} 0 & I \\ -\Omega^2 & 0 \end{bmatrix}$. Obviously, this corresponds to a vectorial harmonic oscillator, and $\exp(X)$ ought to be symplectic. However, it can be easily checked that $A:=I+X/2^n$ does not satisfy $A^T J A=J$ and hence is not symplectic. $\square$

\smallskip

Our idea is to obtain $F_{2,k}$ and $F_{3,k}$ using a modified scaling and squaring strategy, in which the Pad\'{e} approximation is replaced by a symplectic approximation originated from a reversible symplectic integrator (we use Velocity-Verlet). More precisely, suppose $h>0$ is a small constant, then we have the following identity:
\begin{equation}
    \begin{bmatrix} F_{2,k}(H) & G_{2,k,i}(H) \\ 0 & F_{3,k}(H) \end{bmatrix} = \begin{bmatrix} F_{2,k}(h) & G_{2,k,i}(h) \\ 0 & F_{3,k}(h) \end{bmatrix}^{H/h}
\end{equation}
$F_{3,k}(h)$ can be approximated by the following:
\begin{equation}
    \exp \begin{bmatrix} 0 & hI \\ -h\epsilon^{-1}K(q^{slow}_{k'}) & 0 \end{bmatrix}
    \approx \begin{bmatrix} I-\frac{h^2}{2} \epsilon^{-1}K(q^{slow}_{k'}) &
    h\left(I-\frac{h^2}{4} \epsilon^{-1}K(q^{slow}_{k'})\right) \\
    -h\epsilon^{-1}K(q^{slow}_{k'}) &
    I-\frac{h^2}{2} \epsilon^{-1}K(q^{slow}_{k'}) \end{bmatrix} ,
    \label{symplecticCVerlet}
\end{equation}
which can be easily checked to be symplectic thanks to the specific $\mathcal{O}(h^2)$ and $\mathcal{O}(h^3)$ corrections in the above expression.

It is a classical result (global error bound of Velocity-Verlet) that links $F_{3,k}(H)$ with the approximated $F_{3,k}(h)$:
\begin{equation}
    \left\| \exp \begin{bmatrix} 0 & HI \\ -H\epsilon^{-1}K(q^{slow}_{k'}) & 0 \end{bmatrix}
    - \begin{bmatrix} I-\frac{h^2}{2} \epsilon^{-1}K(q^{slow}_{k'}) &
    h\left(I-\frac{h^2}{4} \epsilon^{-1}K(q^{slow}_{k'})\right) \\
    -h\epsilon^{-1}K(q^{slow}_{k'}) &
    I-\frac{h^2}{2} \epsilon^{-1}K(q^{slow}_{k'}) \end{bmatrix}^{H/h} \right\|_2 \leq \epsilon^{-1} C \exp(CH)h^2
\end{equation}
for some constant $C>0$, because the approximation in \eqref{symplecticCVerlet} corresponds to the celebrated Velocity-Verlet integrator with updating rule:
\begin{equation}
\begin{cases}
    x_{i+\frac{1}{2}}=x_i+\frac{h}{2} y_i \\
    y_{i+1}=y_i-h \epsilon^{-1} K(q^{slow}_{k'}) x_{i+\frac{1}{2}} \\
    x_{i+1}=x_{i+\frac{1}{2}}+\frac{h}{2} y_{i+1}
\end{cases}
\end{equation}
for the system $
\begin{cases}
    \dot{x}=y \\
    \dot{y}=-\epsilon^{-1} K(q^{slow}_{k'}) x
\end{cases}$,
which is well-known to have a 2nd-order global error.

We can repeat the same procedure to get an approximation of $F_{2,k}(H)$ by using the following approximated $F_{2,k}(h)$:
\begin{equation}
    \exp \begin{bmatrix} 0 & h\epsilon^{-1}K^T(q^{slow}_{k'}) \\ -hI & 0 \end{bmatrix}
    \approx \begin{bmatrix} I-\frac{h^2}{2} \epsilon^{-1}K^T(q^{slow}_{k'}) &
    h\epsilon^{-1}K^T(q^{slow}_{k'}) \\
    -h\left(I-\frac{h^2}{4} \epsilon^{-1}K^T(q^{slow}_{k'})\right) &
    I-\frac{h^2}{2} \epsilon^{-1}K^T(q^{slow}_{k'}) \end{bmatrix}
    \label{symplecticAVerlet}
\end{equation}

To approximate $G_{2,k,i}(h)$, we follow the result of Lemma \ref{lemma_flowSymplecticity} that in the continuous case $G_{2,k,i}=-J \frac{\partial}{\partial q^{slow}_{k',i}} F_{3,k}$ and let
\begin{equation}
    G_{2,k,i}(h)=-J \partial_i F_{3,k}(h) \approx \begin{bmatrix} h\epsilon^{-1}\frac{\partial}{\partial q^{slow}_{k',i}} K(q^{slow}_{k'}) &
    \frac{h^2}{2}\epsilon^{-1}\frac{\partial}{\partial q^{slow}_{k',i}} K(q^{slow}_{k'}) \\
    -\frac{h^2}{2}\epsilon^{-1}\frac{\partial}{\partial q^{slow}_{k',i}} K(q^{slow}_{k'}) &
    -\frac{h^3}{4}\epsilon^{-1}\frac{\partial}{\partial q^{slow}_{k',i}} K(q^{slow}_{k'})  \end{bmatrix}
    \label{symplecticBVerlet}
\end{equation}

Notice that if (1,0) Pad\'{e} approximation (i.e., 1st-order Taylor expansion) is used, we will get
\begin{equation}
    G_{2,k,i}(h)\approx h M_i = \begin{bmatrix} h\epsilon^{-1}\frac{\partial}{\partial q^{slow}_{k',i}} K(q^{slow}_{k'}) & 0 \\ 0 & 0 \end{bmatrix}
\end{equation}
Naturally, \eqref{symplecticBVerlet} is a higher order correction of this.

$G_{2,k,i}(H)$ will also be accurate: since the accuracy of \eqref{expm} is well established, the higher order corrections that we add in $F_{2,k}(H),F_{3,k}(H),G_{2,k,i}(H)$ will not lead a scheme less accurate. This can immediately be seen in the context of the numerical integration of a stable system, where a local error of $\mathcal{O}(h^2)$ will only lead to a global error of at most $\epsilon^{-1}CHh$ \cite{LaRi56}. We also refer to Appendix A in \cite{MoLo:03} for an analogous error analysis if one prefers to directly work with matrices.

To sum up, the following numerical approximation of $F_{3,k}$ and $G_{2,k,i}$ will simultaneously guarantee symplecticity, accuracy, and efficiency:

\begin{Integrator}
Matrix exponentiation scheme that complements the updating rule of Integrator \ref{extended_SIM}. $n\geq 1$ is an integer controlling the accuracy of the approximation of the matrix exponentials. $k$ is the same index as the one used in Integrator \ref{extended_SIM}, and the following needs to be done for each $k$:

\begin{enumerate}
\item
    Evaluate $K_k:=K(q^{slow}_{k'})$ and $\partial_i K_k:=\frac{\partial}{\partial q^{slow}_{k',i}} K(q^{slow}_{k'})$. Let $h=H/2^n$,
    \begin{equation}
        A_k:=\begin{bmatrix} I-\epsilon^{-1} K_k \frac{h^2}{2} && \epsilon^{-1} K_k h \\ -h (I- \epsilon^{-1} K_k \frac{h^2}{4} ) && I-\epsilon^{-1} K_k \frac{h^2}{2} \end{bmatrix} ,
        \label{Definition_Ak}
    \end{equation}
    \begin{equation}
        C_k:=\begin{bmatrix} I-\epsilon^{-1} K_k \frac{h^2}{2} && h (I- \epsilon^{-1} K_k \frac{h^2}{4} ) \\ -\epsilon^{-1} K_k h && I-\epsilon^{-1} K_k \frac{h^2}{2} \end{bmatrix} ,
        \label{Definition_Ck}
    \end{equation}
    and for $i=1,\ldots,d_s$,
    \begin{equation}
        B_{k,i}:=\begin{bmatrix} \epsilon^{-1} \partial_i K_k h && \epsilon^{-1} \partial_i K_k \frac{h^2}{2} \\ -\epsilon^{-1} \partial_i K_k \frac{h^2}{2} && -\epsilon^{-1} \partial_i K_k \frac{h^3}{4} \end{bmatrix} .
        \label{Definition_Bk}
    \end{equation}
\item
    Let $F_{2,k}^1:=A_k$, $G_{2,k,i}^1:=B_{k,i}$, $F_{3,k}^1:=C_k$, then repetitively apply $\begin{bmatrix} F_{2,k}^{j+1} & G_{2,k,i}^{j+1} \\ 0 & F_{3,k}^{j+1} \end{bmatrix}: = \begin{bmatrix} F_{2,k}^j & G_{2,k,i}^j \\ 0 & F_{3,k}^j \end{bmatrix}^2 = \begin{bmatrix} F_{2,k}^j F_{2,k}^j & F_{2,k}^j G_{2,k,i}^j + G_{2,k,i}^j F_{3,k}^j \\ 0 & F_{3,k}^j F_{3,k}^j \end{bmatrix}$ for $j=1,\ldots,n$.
\item
    Define $F_{2,k}:=F_{2,k}^{n+1}$, $G_{2,k,i}:=G_{2,k,i}^{n+1}$, $F_{3,k}=F_{3,k}^{n+1}$.
    \label{exponential_repeat}
\end{enumerate}
\label{fastMatrixExpScheme}
\end{Integrator}

\begin{Remark}
    The trick for the computational save is that raising to the $2^{n}th$ power is computed by $n$ self multiplications, which is due to the semi-group property of the exponentiation operation. An obvious upper bound to guarantee accuracy is $n\leq C \log \epsilon^{-1}$ (because the error of numerical exponentiation is bounded by $\epsilon^{-1}Ch=\epsilon^{-1}CH/2^n$). In all numerical experiments in this paper, $n=10$ worked well, which is a value much smaller than $\log \epsilon^{-1}$, and this choice of $n$ makes the computation cost of the same order as if $K$ could be diagonalized by a constant orthogonal matrix.
\end{Remark}

\begin{Remark}
    Observe that, for a finite-time simulation, the cost of computing $\phi^3$ numerically with microscopic time-steps  blows up with a speed of $\mathcal{O}(\epsilon^{-1})$, whereas the cost of matrix exponentiations via Integrator \ref{fastMatrixExpScheme} blows up at a maximum speed of $\mathcal{O}(\log \epsilon^{-1})$.
\end{Remark}

Theorem \ref{symplecticity} shows that Integrator \ref{fastMatrixExpScheme} not only ensures $F_{2,k}$ and $F_{3,k}$ to be symplectic, but also guarantees a symplectic approximation to $\phi^3$ (Eq. \ref{phi3}).

Speed-up is obtained because at each step the computation cost is dominated by $2(d_s+1)n$ matrix production operations (of $d_f\times d_f$ matrices), where $n$ is a small integer. If the Coppersmith-Winograd algorithm is used to realize the matrix multiplication operation,  then the time complexity for exponentiation at each step is $n\mathcal{O}(d_f^{2.376})$ (assuming $d_s=\mathcal{O}(1)$; the problem of matrix exponentiation is less difficult otherwise).

\subsection{Fast numerical matrix exponentiation for the symplectic integration of \eqref{eqqasiquadflow}: an alternative}\label{bliblibloublou}
An alternative way to approximate the flow map \eqref{phi3} is to use the slowly varying property of $K$ to generate a symplectic update of the exponential computed at the previous step. The main idea of the method is as follows: given a sequence of matrices $\{X_k\}$ that vary slowly, use the approximation
\begin{equation}
    \exp (X_k)=[\exp (X_k/2^n)]^{2^n}\approx [\exp (X_{k-1}/2^n) \exp ((X_k-X_{k-1})/2^n)]^{2^n}
    \label{fastMatrixExp}
\end{equation}
where $n$ is a preset constant. Again, we use the trick of self-multiplication for computing the $2^{n}th$ power, and efficiency is guaranteed exactly as before.

Accuracy is achieved because, as shown in the following theorem, the approximation error decreases at an exponential rate with respect to $n$.

\begin{theorem} Theorem 5 in \cite{MoLo:03}:
    \begin{equation}
        \| \exp(A+B)-(\exp(A/2^n)\exp(B/2^n))^{2^n} \|_2 \leq 2^{-n-1}e^{\max(\mu(A+B),\mu(A)+\mu(B))} \|[A,B]\|_2
        \label{Lie-Trotter_bound_equation}
    \end{equation}
    where $\mu(X)$ is the maximum eigenvalue of $(X^*+X)/2$, and $[A,B]=AB-BA$ is the canonical Lie bracket.
\end{theorem}

\begin{Remark}[Generality]
    This exponentiation method based on corrections \eqref{fastMatrixExp} is not limited to the integration of \eqref{quasi-quadratic}, but works for repetitive exponentiations of any slowly varying matrix. It would also work for a set of matrices, as long as they could be indexed to ensure a slow variation.
\end{Remark}

For our case, $X_k$ and $A$ are identified with $N$ in Section \ref{gguqirehugqubliafd} at each timestep, and $B$ is identified as the difference in $N$'s between consecutive steps. Since $K(q^{slow})$ (and hence $N$ as well) is changing slowly, $\|B\|_2 \ll \|A\|_2$; furthermore, the calculation of $[A,B]$ (omitted; notice that $B$ is nilpotent) shows that $\|[A,B]\|\ll \|A\|$. Therefore, the error bound here \eqref{Lie-Trotter_bound_equation} is much smaller than that based on scaling and squaring for the same $n$. Consequently, we will be able to further decrease the value of $n$ by a few (not a lot because a decrease in $n$ exponentially increases the error).

The reason that we do not identify $X_k$ and $A$ with $\begin{bmatrix} -N^T & M_i \\ 0 & N \end{bmatrix}$ is due to a consideration of symplecticity in all variables, because otherwise $G_{2,k,i}$, obtained as the upper-right block of the exponential, will not be exactly the derivative of $F_{3,k}$. Instead, we let $G_{2,k,i}=-J \frac{\partial}{\partial q^{slow}_{k',i}} F_{3,k}$, where $F_{3,k}$ is updated from $F_{3,k-1}$ using \eqref{fastMatrixExp}. Taking the derivative, however, incurs additional computation, because $F_{3,k}$ now depends on not only $q^{slow}_k$ but also $q^{slow}_{k-1}$, and therefore $\partial q^{slow}_{k',i} / \partial q^{slow}_{(k-1)',j}$ has to be computed so that a chain rule applies to facilitate the computation. In the end, the computational saving based on updating the exponentiation becomes less significant due to the extra cost in updating $\partial q^{slow}_{k',i} / \partial q^{slow}_{(k-1)',j}$, but the implementation becomes more convoluted. We leave the details to the appendix.

\section{Analysis}
\subsection{Symplecticity}

For a concise writing, we carry out matrix analysis in block forms in this section. Coordinates are ordered as $q^{fast},p^{fast},q^{slow},p^{slow}$, and therefore $\mathbb{J}=\begin{bmatrix} J & 0 \\ 0 & J \end{bmatrix}$ is the coordinate representation of the canonical symplectic 2-form on the full phase space (abusing notations, we use $J:=\begin{bmatrix} 0 & I \\ -I & 0 \end{bmatrix}$ to represent the symplectic 2-form on both the fast subspace (for $q^{fast},p^{fast}$) and the slow subspace (for $q^{slow},p^{slow}$); this should not affect the clarity of the analysis). We also recall that a map $x\mapsto \phi(x)$ is symplectic if and only if $\phi'(x)^T \mathbb{J} \phi'(x)=\mathbb{J}$ or $\phi'(x)^T J \phi'(x)=J$ for all $x$'s (depending on whether $x$ represents all variables or only slow or fast variables).

\begin{lemma}
The numerical approximation to $\phi^3$ given by \eqref{phi3numerical} is symplectic on all variables if and only if $F_{3,k}$ is symplectic and, for $i=1,\ldots,d_s$, $G_{2,k,i}=-J \frac{\partial F_{3,k}}{\partial q_{k',i}^{slow}}$ (note that for a fixed $i$, $G_{2,k,i}$, $\frac{\partial F_{3,k}}{\partial q_{k',i}^{slow}}$ and $J$ are  $d_f \times d_f$ matrices).
\label{lemma_flowSymplecticity}
\end{lemma}
\begin{proof}
    For conciseness and convenient reading, write $q_{k'}^{fast}$ and $p_{k'}^{fast}$ as $q_f$ and $p_f$, $\partial/\partial q^{slow}_{k',i}$ as $\partial_i$, and $G_{2,k,i}$ and $F_{3,k}$ as $G_{2,i}$ and $F_3$ in this proof.

    The Jacobian of the numerical approximation to $\phi^3:q_{k'},p_{k'} \mapsto q_{k+1},p_{k+1}$ given by \eqref{phi3numerical} can be computed as:
    \begin{align}
        A=\begin{array}{c|c} F_{3} &
            \begin{array}{c|c} \partial_1 F_{3} \begin{pmatrix} q_f \\ p_f \end{pmatrix} ~ \cdots ~ \partial_{d_s} F_{3} \begin{pmatrix} q_f \\ p_f \end{pmatrix} & \begin{pmatrix} 0 & \cdots & 0 \\ 0 & \cdots & 0 \end{pmatrix} \end{array} \\
            \hline
            \begin{array}{c} \begin{pmatrix} 0 & 0 \\ \vdots & \vdots \\ 0 & 0 \end{pmatrix} \\ \hline -\begin{pmatrix} q_f^T & p_f^T \end{pmatrix} F_{3}^T G_{2,1} \\ \vdots \\ -\begin{pmatrix} q_f^T & p_f^T \end{pmatrix} F_{3}^T G_{2,d_s} \end{array} &
            \begin{array}{c|c} I & 0 \\ \hline -* & I \end{array}
        \end{array}
    \end{align}
    where $(*)_{i,j}=\frac{1}{2} [q_f;p_f]^T \partial_j (F_{3}^T G_{2,i}) [q_f; p_f]$, and the $0$'s in the upper right block, the lower left block, and the lower right block respectively corresponds to $d_f$-by-$1$, $1$-by-$d_f$, and $d_s$-by-$d_s$ zero matrices. Notice that we have $F_{3}^T G_{2,i}$ in the lower left block because $F_{3}^T G_{2,i}$ is symmetric (their exact values satisfy this because of \eqref{symmetricMatrix}, and their numerical approximations satisfy this because of Lemma \ref{lemma_symmetricMatrix}).

    Symplecticity is equivalent to $A^T\mathbb{J}A=\mathbb{J}$, whose left hand side writes out to be
    \noindent
    \begin{align}
        &~ A^T\mathbb{J}A =\begin{array}{c|c} F_3^T J &
                \begin{array}{c|c} G_{2,1}^T F_{3} \begin{pmatrix} q_f \\ p_f \end{pmatrix} ~ \cdots ~ G_{2,d_s}^T F_{3} \begin{pmatrix} q_f \\ p_f \end{pmatrix} & \begin{pmatrix} 0 & \cdots & 0 \\ 0 & \cdots & 0 \end{pmatrix} \end{array} \\
                \hline
                \begin{array}{c}
                    \begin{pmatrix} q_f^T & p_f^T \end{pmatrix} \partial_1 F_3^T J \\
                    \vdots \\
                    \begin{pmatrix} q_f^T & p_f^T \end{pmatrix} \partial_{d_s} F_3^T J \\
                    \hline
                    \begin{pmatrix} 0 & 0 \\ \vdots & \vdots \\ 0 & 0 \end{pmatrix}
                \end{array}
                &
                \begin{array}{c|c} *^T & I \\ \hline -I & 0 \end{array}
            \end{array} \times A \nonumber\\
            &= \begin{array}{c|c} F_3^T J F_3 + 0 &
                \begin{array}{c|c}
                    (F_{3}^T J \partial_1 F_{3}+G_{2,1}^T F_{3}) \begin{pmatrix} q_f \\ p_f \end{pmatrix} ~ \cdots ~ (F_{3}^T J \partial_{d_s} F_{3}+G_{2,d_s}^T F_{3}) \begin{pmatrix} q_f \\ p_f \end{pmatrix} &
                    \begin{pmatrix} 0 & \cdots & 0 \\ 0 & \cdots & 0 \end{pmatrix}
                \end{array} \\
               \hline
                \triangle &
                \begin{array}{c|c}
                    \left( [q_f;p_f]^T \partial_i F_3^T J \partial_j F_3 [q_f;p_f] \right)_{i=1,\ldots,d_s;j=1,\ldots,d_s}
                    & 0 \\ \hline 0 & 0
                \end{array} +
                \begin{array}{c|c} -*^T+* & I \\ \hline -I & 0 \end{array}
            \end{array}
    \end{align}
    where $\triangle$ is naturally negative the transpose of the upper-right block because $A^T \mathbb{J} A$ is skew-symmetric for any $A$.

    This is equal to $\mathbb{J}$ if and only if the upper-left block and the bottom-right block are both $J$ and the upper-right block and the bottom-left block are both zero. The requirement on upper-left block is
    \begin{equation}
        F_3^T J F_3 = J
        \label{gqagrdsfkljaldg}
    \end{equation}
    By the arbitrariness of $q_f$ and $p_f$, the requirement on upper-right and bottom-left blocks translates to:
    \begin{equation}
        F_3^T J \partial_i F_3 + G_{2,i}^T F_3 = 0
    \end{equation}
    which further simplifies to
    \begin{equation}
        G_{2,i}=-J \partial_i F_3
        \label{iffSymplectic}
    \end{equation}
    because $F_3^T J \partial_i F_3 = \partial_i(F_3^T J F_3)-\partial_i F_3^T J F_3=-\partial_i F_3^T J F_3$, $F_3$ is invertible due to \eqref{gqagrdsfkljaldg}, and $J^T=-J$.

    The bottom-right block needs to be $J$, and this requirement is equivalent to
    \begin{equation}
        [q_f;p_f]^T \left(\partial_i F_3^T J \partial_j F_3+\frac{1}{2}\partial_i(F_3^T G_{2,j})-\frac{1}{2}\partial_j(F_3^T G_{2,i})\right)[q_f;p_f]=0
        \label{o7tyhoquilbvbilrfquefqhfi}
    \end{equation}
    By \eqref{iffSymplectic}, the above left hand side rewrites as
    \begin{align}
        &~ [q_f;p_f]^T \left( \partial_i F_3^T J \partial_j F_3
        -\frac{1}{2}\partial_i F_3^T J \partial_j F_3-\frac{1}{2} F_3^T J \partial_i \partial_j F_3
        +\frac{1}{2}\partial_j F_3^T J \partial_i F_3+\frac{1}{2} F_3^T J \partial_j \partial_i F_3 \right) [q_f;p_f] \nonumber\\
        &= [q_f;p_f]^T \left( \frac{1}{2}\partial_i F_3^T J \partial_j F_3 \right) [q_f;p_f] + [q_f;p_f]^T \left( \frac{1}{2}\partial_j F_3^T J \partial_i F_3 \right) [q_f;p_f]
    \end{align}
    Since what are summed up above are just two real numbers, the second number remains the same after taking its transpose, which due to $J^T=-J$ yields
    \begin{equation}
        [q_f;p_f]^T \left( \frac{1}{2}\partial_j F_3^T J \partial_i F_3 \right) [q_f;p_f]=-[q_f;p_f]^T \left( \frac{1}{2}\partial_i F_3^T J \partial_j F_3 \right) [q_f;p_f]
    \end{equation}
    Therefore, \eqref{o7tyhoquilbvbilrfquefqhfi} does hold. $\square$
\end{proof}

\begin{lemma}
    In Integrator \ref{fastMatrixExpScheme}, all $A_k$ and $C_k$ are symplectic; moreover, all $F_{2,k}$ and $F_{3,k}$ are symplectic, too.
    \label{lemma_symplecticity}
\end{lemma}
\begin{proof}
    Straightforward computation using \eqref{Definition_Ak} and \eqref{Definition_Ck} shows that $A_k^T J A_k=J$ and $C_k^T J C_k=J$. Moreover, since the product of symplectic matrices is symplectic, all $F_{2,k}$ and $F_{3,k}$, being powers of $A_k$ and $C_k$, are symplectic. $\square$
\end{proof}

\begin{lemma}
    In Integrator \ref{fastMatrixExpScheme}, $A_k^T C_k=I$ (and equivalently $C_k A_k^T=I$) for all $k$; moreover, $F_{2,k}^T F_{3,k}=I$ (and equivalently $F_{3,k} F_{2,k}^T=I$).
    \label{lemma_reversible}
\end{lemma}
\begin{proof}
    Straightforward computation using \eqref{Definition_Ak} and \eqref{Definition_Ck} shows that $A_k^T C_k=I$. Therefore, $(A_k A_k)^T C_k C_k=A_k^T I C_k=I$, and by induction $(A_k^{2^n})^T C_k^{2^n}=I$, i.e., $F_{2,k}^T F_{3,k}=I$.
\end{proof}

\begin{lemma}
In Integrator \ref{fastMatrixExpScheme}, $B_{k,i}=-J \frac{\partial}{\partial q^{slow}_{k',i}} C_k$ for all $k$ and $i$, and $G_{2,k,i}=-J \frac{\partial}{\partial q^{slow}_{k',i}} F_{3,k}$ for all $k$ and $i$.
\label{lemma_flowPrimaryFunction}
\end{lemma}
\begin{proof}
    Use the short-hand notation $\partial_i := \frac{\partial}{\partial q^{slow}_{k,i}}$. Straightforward computation using \eqref{Definition_Bk} and \eqref{Definition_Ck} shows that $B_{k,i}=-J \partial_i C_k$ for all $k$ and $i$.

    Since $\begin{bmatrix} F_{2,k} & G_{2,k,i} \\ 0 & F_{3,k} \end{bmatrix} = \begin{bmatrix} A_k & B_{k,i} \\ 0 & C_k \end{bmatrix}^{2^n}$ for all $i$, by induction, it is only necessary to prove that $G_{2,k,i}=-J\partial_i F_{3,k}$ when $n=1$. In this case, $G_{2,k,i}=A_k B_{k,i}+B_{k,i} C_k$ and $F_{3,k}=C_k C_k$, and the equality can be proved by the following:

    Because $B_{k,i}=-J \partial_i C_k$, $C_k^T A_k=I$ (Lemma \ref{lemma_reversible}) and $J=C_k^T J C_k$ (Lemma \ref{lemma_symplecticity}), we have
    \begin{equation}
        C_k^T A_k B_{k,i} = -C_k^T J C_k \partial_i C_k
    \end{equation}
    Since symplectic matrix is nonsingular, this is
    \begin{equation}
        A_k B_{k,i} = -J C_k \partial_i C_k
    \end{equation}
    Adding $B_{k,i} C_k=-J \partial_i C_k C_k$, we have
    \begin{equation}
        A_k B_{k,i}+B_{k,i} C_k=-J\partial_i (C_k C_k)
    \end{equation}
    Hence, the induction works. $\square$
\end{proof}

\begin{lemma}
    In Integrator \ref{fastMatrixExpScheme}, $C_k^T B_{k,i}=B_{k,i}^T C_k$ for all $k$ and $i$.
\label{lemma_symmetricMatrix_sub}
\end{lemma}
\begin{proof}
    This can be shown by straightforward computation using \eqref{Definition_Bk} and \eqref{Definition_Ck}.
\end{proof}

\begin{lemma}
    In Integrator \ref{fastMatrixExpScheme}, $F_{3,k}^T G_{2,k,i}=G_{2,k,i}^T F_{3,k}$ for all $k$ and $i$.
\label{lemma_symmetricMatrix}
\end{lemma}
\begin{proof}
    By Lemma \ref{lemma_symmetricMatrix_sub}, $C_k^T B_{k,i}=B_{k,i}^T C_k$ for all $k$ and $i$. By Lemma \ref{lemma_reversible}, $A_k^T C_k=I$ and $C_k^T A_k=I$.

    Since $\begin{bmatrix} F_{2,k} & G_{2,k,i} \\ 0 & F_{3,k} \end{bmatrix} = \begin{bmatrix} A_k & B_{k,i} \\ 0 & C_k \end{bmatrix}^{2^n}$ for all $i$, by induction, it is only necessary to prove that $F_{3,k}^T G_{2,k,i}=G_{2,k,i}^T F_{3,k}$ when $n=1$. In this case, $G_{2,k,i}=A_k B_{k,i}+B_{k,i} C_k$ and $F_{3,k}=C_k C_k$, and this equality can be proved upon observing for all $i$:
    \begin{align}
        &~ C_k^T C_k^T (A_k B_{k,i} + B_{k,i} C_k)
        = C_k^T B_{k,i} + C_k^T C_k^T B_{k,i} C_k
        = B_{k,i}^T C_k + C_k^T B_{k,i}^T C_k C_k \nonumber\\
        &= B_{k,i}^T A_k^T C_k C_k + C_k^T B_{k,i}^T C_k C_k
        = (A_k B_{k,i} + B_{k,i} C_k)^T C_k C_k
    \end{align}
\end{proof}

\begin{theorem}
    The proposed method (Integrator \ref{extended_SIM}+\ref{fastMatrixExpScheme}) is symplectic on all variables.
    \label{symplecticity}
\end{theorem}
\begin{proof}
    By Lemma \ref{lemma_symplecticity}, \ref{lemma_flowPrimaryFunction}, \ref{lemma_flowSymplecticity}, and \ref{lemma_symmetricMatrix}, the numerical approximation to $\phi^3$ given by \eqref{phi3numerical} is symplectic on all variables.

    The flow given by \eqref{phi12} is symplectic on all variables as well, because it is the composition of $\phi^1$ and $\phi^2$, which respectively correspond to Hamiltonians $\mathcal{H}_1(q^{fast},p^{fast},q^{slow},p^{slow})=[p^{slow}]^2/2$ and $\mathcal{H}_2(q^{fast},p^{fast},q^{slow},p^{slow})=V(q^{fast},q^{slow})$, and hence both are symplectic.

    Consequently, the proposed method, which composes \eqref{phi12} and \eqref{phi3numerical}, is symplectic. $\blacksquare$
\end{proof}

\label{SectionAnalysisSymplecticity}

\subsection{Uniform convergence}

This integrator is convergent due to splitting theory \cite{Trotter:59}, i.e., the global error on $q^{slow},q^{fast},p^{slow},p^{fast}$ is bounded by $\epsilon^{-1} C H$ for some constant $C>0$ in Euclidean norm.

Moreover, this integrator is uniformly convergent in $q$ under typical or reasonable assumptions, and hence $H$ can be chosen independently from $\epsilon$ for stable and accurate integration.

\begin{Condition}
    We will prove a uniform bound of the global error on position for Integrator \ref{extended_SIM} under the following (classical) conditions:
    \begin{enumerate}
        \item Regularity:
            In the integration domain of interest, $\nabla V(\cdot)$ is bounded and Lipschitz continuous with coefficient $L$, i.e. $\| \nabla V(a)-\nabla V(b) \|_2 \leq L \| a-b \|_2$.
            \label{LipschitzCondition}
        \item Stability and bounded energy:
            For a fixed $T$ and $t<T$, denote by $x(t)=(q(t),p(t))$ the exact solution to \eqref{fullSystem}, and by $x_t=(q_t,p_t)$ the discrete numerical trajectory given by Integrator \ref{extended_SIM}, then $\|x(t)\|_2^2 \leq C$, $\|x_t\|_2^2 \leq C$, $|\mathcal{H}(q(t),p(t))| \leq C$ and $|\mathcal{H}(q_t,p_t)| \leq C$ for some constant $C$ independent of $\epsilon^{-1}$ but dependent on initial condition $\| \begin{bmatrix} q_0 \\ p_0 \end{bmatrix} \|_2^2$ and possibly $T$ as well.
            \label{boundednessCondition}
    \end{enumerate}
    \label{Condition1}
\end{Condition}

\begin{Condition}[Slowly varying frequencies:]
    Consider the solution $q(s),p(s)$ up to time $s<=H$ to the system
    \begin{equation}
    \left\{ \begin {array} {rcl}
        dq^{fast} &=& p^{fast} dt \\
        dq^{slow} &=& p^{slow} dt \\
        dp^{fast} &=& -\partial V / \partial q^{fast}(q^{fast},q^{slow})dt - \epsilon^{-1} K(q^{slow}) q^{fast} dt \\
        dp^{slow} &=& -\partial V / \partial q^{slow}(q^{fast},q^{slow})dt - \epsilon^{-1} \frac{1}{2} [q^{fast}]^T \nabla K(q^{slow}) q^{fast} dt
    \end {array} \right. ,
    \label{SIMvaryK}
    \end{equation}
    with initial condition $q(0),p(0)$ in the domain of interest that satisfies bounded energy.
    Assume that $q^{fast}$ can be written as
    \begin{equation}
        \textbf{Q}(t) \sum_{i=1}^{d_f} \vec{\textbf{e}}_i \sqrt{\epsilon} a_i(t) \cos [\sqrt{\epsilon^{-1}} \theta_i(t)+\phi_i]
        \label{qfastForm}
    \end{equation}
    where $\textbf{Q}(t)$ is a slowly varying matrix (i.e., $Q_{ij}(t)\in C^1([0,H])$  and there exists a $C$ independent of $\epsilon^{-1}$ such that $\|\textbf{Q}(t)\|\leq C$ and $\|\dot{\textbf{Q}}(t)\|\leq C$ for all $t\in[0,H]$), indicating a slowly varying diagonalization frame, $d_f$ is the dimension of the fast variable, $\vec{\textbf{e}}_i$ are standard vectorial basis of $\mathbb{R}^{d_f}$, $a_i(t)$'s are slowly varying amplitudes (in the same sense as for $\textbf{Q}(t)$), $\theta_i(t)$'s are non-decreasing and slowly varying in the sense that $\theta_i(t)\in C^2 ([0,H])$, $|\ddot{\theta}_i(t)| \leq C$, $|\theta_i(t)| \leq C$, and $C_1 \leq \dot{\theta}_i(t)\leq C_2$ for some $C>0,C_1>0,C_2>0$ independent of $\epsilon^{-1}$, and $\phi_i$'s are such that $\theta_i(0)=0$.
    \label{slowVarying}
\end{Condition}

\begin{Remark}
    In the case of constant frequencies ($K(\cdot)$ being a constant) and no slow drift ($V(\cdot)$ being a constant), we  have $q^{fast}=\textbf{Q} \sum_{i=1}^{d_f} \vec{\textbf{e}}_i \sqrt{\epsilon} a_i \cos [\sqrt{\epsilon^{-1}}\omega_i t+\phi_i]$ (the amplitude is $\mathcal{O}(\sqrt{\epsilon})$ because of bounded energy).
     When $K$ is not a constant, Condition \ref{slowVarying} is supported by an asymptotic expansion of $q^{fast}$.
     In particular, to the leading order in $\epsilon$, we have $\dot{\theta}_i(t)= \omega_i(t)$ where the $\omega_i^2(t)$ are the eigenvalues
     of $K(q^{slow}_s)$. The rigorous justification of this asymptotic expansion for $d_f>1$ is beyond the scope of this paper.
\end{Remark}

\begin{lemma}
    If Condition \ref{slowVarying} holds, there exists $C_1>0,C_2>0$ independent of $\epsilon^{-1}$ such that
    \begin{equation}
        \| \int_0^H f(t) q^{fast}(t) dt \| \leq \epsilon \left( C_1 \max_{0 \leq s \leq H} \|f(s)\| + C_2 H \max_{0 \leq s \leq H} \|\dot{f}(s)\| + \mathcal{O}(H^2) \right)
        \label{averagingLemmaEq}
    \end{equation}
    for arbitrary matrix valued function $f\in C^1([0,H])$ that satisfies $f(0)=0$.
    \label{averagingLemma}
\end{lemma}

\begin{proof}
    Recall the form of $q^{fast}$ in Condition \ref{slowVarying}. It is sufficient to prove that for all $i$'s the $i$-th component of $q^{fast}$ satisfies \eqref{averagingLemmaEq}, whereas the $i$-th component writes as:
    \begin{equation}
        \sqrt{\epsilon} \sum_{j=1}^{d_f} Q_{ij}(t)a_j(t)\cos[\sqrt{\epsilon^{-1}}\theta_i(t)+\phi_i]
    \end{equation}
    Furthermore, since summation commutes with integral and therefore will only introduce a factor of $d_f$ on the bound, it is sufficient to prove \eqref{averagingLemmaEq} for $q^{fast}=\sqrt{\epsilon} Q_{ij}(t)a_j(t)\cos[\sqrt{\epsilon^{-1}}\theta_i(t)+\phi_i]$. On this token, we could assume that we are in the 1D case and absorb $Q(t)$ into $a_j(t)$.

    Similarly, slowly varying $a_i(t)$ can be absorbed into the test function $f(s)$, and doing so will only change the constants on the right hand side. Therefore, it will be sufficient to prove that:
    \begin{equation}
        \left| \int_0^H \sqrt{\epsilon} \cos[\sqrt{\epsilon^{-1}} \theta(t)+\phi] f(t) dt \right| \leq \epsilon \left( C_1 \max_{0 \leq s \leq H} |f(s)| + C_2 H \max_{0 \leq s \leq H} |f'(s)| + \mathcal{O}(H^2) \right)
        \label{gadgfsa}
    \end{equation}
    for a scalar valued function $f\in C^1([0,H])$ that satisfies $f(0)=0$.

    By Condition \ref{slowVarying}, $\theta$ is strictly increasing. If we write $\tau=\theta(t)$, there will be a $\theta^{-1}$ such that $t=\theta^{-1}(\tau)$. With time transformed to the new variable $\tau$, the integral on the left hand side of \eqref{gadgfsa} is equal to
    \begin{equation}
        \int_0^{\theta(H)} \sqrt{\epsilon} \cos [\sqrt{\epsilon^{-1}}\tau+\phi] f(\theta^{-1}(\tau)) \frac{d\theta^{-1}}{d\tau} (\tau) \,d\tau
    \end{equation}
    By integration by parts, this is (since $f(0)=0$)
    \begin{equation}
        -\epsilon \sin[\sqrt{\epsilon^{-1}}H+\phi] f(H) \frac{1}{\dot{\theta}(H)}+\epsilon \int_0^{\theta(H)} \sin[\sqrt{\epsilon^{-1}}\tau+\phi] \left[ \frac{df}{dt} \left( \frac{d\theta^{-1}}{d\tau} \right)^2 + f(\theta^{-1}(\tau)) \frac{d^2 \theta^{-1}}{d\tau^2} (\tau) \right]
    \end{equation}
    Because $\ddot{\theta}\leq C$, $\omega-CH \leq \dot{\theta}\leq \omega+CH$, where $\omega:=\dot{\theta}(0) \geq C_1 > 0$. Together with $\frac{d\theta^{-1}}{d\tau}=\frac{1}{\dot{\theta}}$, we have $\frac{d\theta^{-1}}{d\tau} = 1/\omega+\mathcal{O}(H)$. Similarly, we also have
    \begin{equation}
        \frac{d^2 \theta^{-1}}{d\tau^2}
        =\frac{d}{d\tau}\frac{1}{\dot{\theta}(t)}
        =\frac{dt}{d\tau}\frac{d}{dt}\frac{1}{\dot{\theta}(t)}
        =-\frac{1}{\dot{\theta}(t)^3}\ddot{\theta}(t)
        =\mathcal{O}(1)
    \end{equation}
    It is easy to show that $\theta(H)=\mathcal{O}(H)$. Together with $\sin(\cdot)$ being $\mathcal{O}(1)$, the left hand side in \eqref{gadgfsa} is bounded by
    \begin{equation}
        \epsilon f(H) \mathcal{O}(1)+\epsilon \mathcal{O}(H) \left( \mathcal{O}(1) \max_{0 \leq s \leq H} |\dot{f}(s)| + \mathcal{O}(1) \max_{0 \leq s \leq H} |f(s)| \right) \leq \epsilon \left( \mathcal{O}(1) \max_{0 \leq s \leq H} |f(s)| + \mathcal{O}(H) \max_{0 \leq s \leq H} |\dot{f}(s)| \right)
    \end{equation}
\end{proof}

\begin{theorem}
    If Conditions \ref{Condition1} and \ref{slowVarying} hold, the proposed method (Integrator \ref{extended_SIM}) for system \eqref{fullSystem} has a uniform global error of $\mathcal{O}(H)$ in $q$, given a fixed total simulation time $T=NH$:
    \begin{equation}
        \| q(T)-q_T \|_2 \leq C H
        \label{globalError}
    \end{equation}
    where $q(T),p(T)$ is the exact solution and $q_T,p_T$ is the numerical solution; $C$ is a positive constant independent of $\epsilon^{-1}$ but dependent on simulation time $T$, scaleless elasticity matrix $K$, slow potential energy $V(\cdot)$ and initial condition $\| \begin{bmatrix} q_0 \\ p_0 \end{bmatrix} \|_2$.
\label{UniformConvergence}
\end{theorem}

\begin{proof}
    Let  $\tilde{K}$ be a constant matrix and consider the following system:
    \begin{equation}
        \left\{ \begin {array} {rcl}
        d\tilde{q}^{fast} &=& \tilde{p}^{fast} dt\\
        d\tilde{q}^{slow} &=& \tilde{p}^{slow} dt \\
        d\tilde{p}^{fast} &=& -\partial V / \partial q^{fast}(\tilde{q}^{fast},\tilde{q}^{slow})dt - \epsilon^{-1} \tilde{K} \tilde{q}^{fast} dt \\
        d\tilde{p}^{slow} &=& -\partial V / \partial \tilde{q}^{slow}(\tilde{q}^{fast},\tilde{q}^{slow})dt
        \end {array} \right. ,
        \label{SIM}
    \end{equation}
    Integrator \ref{extended_SIM}, applied to the  system \eqref{SIM} under Condition \ref{Condition1}, has been shown in \cite{SIM1} to be uniformly convergent in ``energy norm'' (or equivalently, uniformly convergent on position and non-uniformly convergent on momentum).
     Recall that the ``energy norm'' was defined in \cite{SIM1} to be
     \begin{equation}\label{energyNormor}
     \| [\tilde{q}, \tilde{p}] \|_E = \sqrt{\tilde{q}^T \tilde{q}+\epsilon \tilde{p}^T \tilde{K}^{-1} \tilde{p}},
      \end{equation}
      but in fact $\tilde{K}^{-1}$ is not important because it is just $\mathcal{O}(1)$, and the following definition would also work for the proof there:
    \begin{equation}
        \| [\tilde{q}, \tilde{p}] \|_E = \sqrt{\tilde{q}^T \tilde{q}+\epsilon \tilde{p}^T \tilde{p}}
        \label{energyNorm}
    \end{equation}
    Observe that,  \eqref{energyNormor} is proportional to the physical energy of $\sqrt{\epsilon}(K^{-1/2}q,K^{-1/2}p)$.

    The system considered here, however, is \eqref{SIMvaryK}.  To prove uniform convergence for \eqref{SIMvaryK}, it is sufficient to show that (i) a $\delta$ difference between two trajectories of \eqref{SIM} in energy norm leads to a difference of $\delta(1+CH)$ in energy norm after a time step $H$ (ii) trajectories of \eqref{SIM} and \eqref{SIMvaryK} starting at the same point remain at at a distance at most $\mathcal{O}(H^2)$ in energy norm after time $H$, i.e., a 2nd order uniform local error. (i) was shown by Lemma 6.5 in \cite{SIM1}, and we will now prove (ii).

    We can assume without loss of generality that we start at time 0, and let $\tilde{K}=K(q^{slow}(0))$, $q^{fast,slow}(0)=\tilde{q}^{fast,slow}(0)$ (where $q^{fast,slow}=(q^{fast},q^{slow})$) and $p^{fast,slow}(0)=\tilde{p}^{fast,slow}(0)$. We first let $x=\tilde{q}^{fast}-q^{fast}$ and $y=\tilde{p}^{fast}-p^{fast}$, and proceed to bound $x$ and $y$:

    The evolutions of $x$ and $y$ follow from
    \begin{equation}
    \begin{cases}
        \dot{x}&=y \\
        \dot{y}&=-\left(\frac{\partial V}{\partial q^{fast}}(\tilde{q})-\frac{\partial V}{\partial q^{fast}}(q)\right)-\epsilon^{-1}\left( \tilde{K}\tilde{q}^f-K(q^{slow})q^{fast} \right)
    \end{cases}
    \end{equation}
    Writing $f_1=-\left(\frac{\partial V}{\partial q^{fast}}(\tilde{q})-\frac{\partial V}{\partial q^{fast}}(q)\right)$ and $f_2=(\tilde{K}-K(q^{slow}))q^{fast}$, we have
    \begin{equation}
    \begin{cases}
        \dot{x}&=y \\
        \dot{y}&=f_1-\epsilon^{-1}\tilde{K}x-\epsilon^{-1}f_2
    \end{cases}
    \end{equation}
    If we let $B(t)=\exp \left( \begin{bmatrix} 0 & I \\ -\epsilon^{-1}\tilde{K} & 0 \end{bmatrix} t \right)$, we will have
    \begin{equation}
        \begin{bmatrix} x(t) \\ y(t) \end{bmatrix} = B(t) \begin{bmatrix} x(0) \\ y(0) \end{bmatrix} + \int_0^t B(t-s) \begin{bmatrix} 0 \\ f_1-\epsilon^{-1}f_2 \end{bmatrix} \, ds
    \end{equation}
    The first term on the right hand side drops off because $x(0)=0$ and $y(0)=0$ by definition.

    Since $\tilde{K}$ is a constant matrix, it is sufficient to diagonalize it and treat each diagonal element individually. Hence, assume without loss of generality that we are in the 1D case. Then $B(s)=\begin{bmatrix} \cos(\sqrt{\epsilon^{-1}\tilde{K}} s) & \sin(\sqrt{\epsilon^{-1}\tilde{K}} s)/\sqrt{\epsilon^{-1}\tilde{K}} \\ -\sqrt{\epsilon^{-1}\tilde{K}} \sin(\sqrt{\epsilon^{-1}\tilde{K}} s) & \cos(\sqrt{\epsilon^{-1}\tilde{K}} s) \end{bmatrix}$. As a consequence,
    \begin{equation}
        y(t)=\int_0^t \cos[\sqrt{\epsilon^{-1}\tilde{K}}(t-s)] \left[ f_1 - \epsilon^{-1} (\tilde{K}-K(q^{slow}))q^{fast} \right] ds
    \end{equation}

    By Lipschitz continuity of $V$ (Item \ref{LipschitzCondition} of Condition \ref{Condition1}), we will have
    \begin{equation}
        |f_1(t)| \leq L |x(t)| = L |\int_0^t y(s) ds| = \mathcal{O}(t)
    \end{equation}
    The first inequality holds because $f_1$ is the difference between partial derivatives of $V$, which could be bounded by the difference between full derivatives. The last equality holds because $y=p-\tilde{p}$ is bounded due to the fact that $[q(s),p(s)]$ and $[\tilde{q}(s), \tilde{p}(s)]$ are bounded (Item \ref{boundednessCondition} of Condition \ref{Condition1}). Consequently, we have
    \begin{equation}
        \left|\int_0^t \cos[\sqrt{\epsilon^{-1}\tilde{K}}(t-s)] f_1 \, ds \right| \leq \int_0^t |f_1| = \mathcal{O}(t^2)
        \label{ghfqaget}
    \end{equation}

    In order to bound $\int_0^t \cos[\sqrt{\epsilon^{-1}\tilde{K}}(t-s)] \left[ \epsilon^{-1} (\tilde{K}-K(q^{slow}))q^{fast} \right] ds$, we use Lemma \ref{averagingLemma} (with the choice of $f=\tilde{K}-K(q^{slow})$). Indeed, $\cos[\sqrt{\epsilon^{-1}\tilde{K}}(t-s)]$ can be absorbed into $q^{fast}(s)=\sqrt{\epsilon} \cos[\sqrt{\epsilon^{-1}} \theta(s)+\phi]$: due to an equality $2\cos(A)\cos(B)=\cos(A+B)+\cos(A-B)$, $\theta$ will be just added by $\pm \sqrt{\tilde{K}}$ and $\phi$ will have a new constant value, neither of which will violate Condition \ref{slowVarying}.

    For $f$, we clearly have $f=0$ at $s=0$. By mean value theorem, there is a $\xi_s$ such that $f(s)=K\circ q^{slow} (0)-K\circ q^{slow} (s)=\frac{d K\circ q^{slow}}{dt} (\xi_s) \cdot s$, and therefore $f(s)=\mathcal{O}(s)$. Similarly, $\dot{f}(s)=\mathcal{O}(1)$. Plotting these two bounds in Lemma \ref{averagingLemma}, we obtain
    \begin{equation}
        \left| \int_0^t \cos[\sqrt{\epsilon^{-1}\tilde{K}}(t-s)] \left[ \epsilon^{-1} (\tilde{K}-K(q^{slow}))q^{fast} \right] ds \right| = \mathcal{O}(t)
    \end{equation}

    Putting this together with \eqref{ghfqaget}, we arrive in $y(t)=\mathcal{O}(t)$, and $x(t)=\int_0^t y(s) \, ds=\mathcal{O}(t^2)$ follows.

    Next, we bound $y$: since
    \begin{equation}
        \left| \int_0^t \cos[\sqrt{\epsilon^{-1}\tilde{K}}(t-s)] \left[ \epsilon^{-1} (\tilde{K}-K(q^{slow}))q^{fast} \right] ds \right| = \left| \int_0^t \cos[\ldots] \epsilon^{-1} \mathcal{O}(s) \sqrt{\epsilon} \mathcal{O}(1) \cos[\ldots] \, ds \right|
        = \epsilon^{-1/2} \mathcal{O}(t^2)
    \end{equation}
    we have $y(t)=\epsilon^{-1/2} \mathcal{O}(t^2)$. Together with $x(t)=\mathcal{O}(t^2)$, this is equivalent to $\|[x,y]\|_E=\mathcal{O}(t^2)$.

    Similarly, we can bound $q^{slow}-\tilde{q}^{slow}$ and $p^{slow}-\tilde{p}^{slow}$. Let $x^s=q^{slow}-\tilde{q}^{slow}$ and $y^s=p^{slow}-\tilde{p}^{slow}$, then we have:
    \begin{equation}
    \begin{cases}
        \dot{x}^s &=y^s \\
        \dot{y}^s &=-\left(\frac{\partial V}{\partial q^{slow}}(\tilde{q})-\frac{\partial V}{\partial q^{slow}}(q)\right)-\epsilon^{-1}\frac{1}{2}[q^{fast}]^T \nabla K(q^{slow}) q^{fast}
    \end{cases}
    \end{equation}
    Analogous to before, the first term on the right hand side of the $y^s$ dynamics is $\mathcal{O}(t)$. Since $q^{fast}=\mathcal{O}(\epsilon^{1/2})$, the second term on the right hand side is $\mathcal{O}(1)$. Therefore, $\dot{y}^s=\mathcal{O}(1)$, $y(t)=y(0)+\mathcal{O}(t)=\mathcal{O}(t)$, and $x(t)=x(0)+\int_0^t y(s) \, ds=\mathcal{O}(t^2)$. For our purpose of fast integration, we use a big timestep $H\geq \sqrt{\epsilon}$, and hence $y(H)=\mathcal{O}(H) \leq \epsilon^{-1/2} \mathcal{O}(H^2)$ (notice that if $H<\sqrt{\epsilon}$, we do not even need to prove uniform convergence, because the non-uniform error bound that is guaranteed by Lie-Trotter splitting theory is already very small).

    $\mathcal{O}(H^2)$ and $\epsilon^{-1/2} \mathcal{O}(H^2)$ bounds on separations of slow position and slow momentum imply a $\mathcal{O}(t^2)$ uniform bound in energy norm (analogous to that of the fast degrees of freedom). This demonstrates a 2nd-order uniform local error on all variables in energy norm, and therefore concludes the proof.
\end{proof}

\begin{Remark}
    Unlike \eqref{globalError}, a global bound on the error of momentum will not be uniform. The error propagation is quantified in energy norm, and in 2-norm we will only have $\epsilon^{-1/2}\mathcal{O}(H^2)$ local error and $\epsilon^{-1/2}\mathcal{O}(H)$ global error on momentum. In fact, Integrator \ref{extended_SIM} applied to the constant frequency system \eqref{SIM} is  non-uniformly convergent on momentum \cite{SIM1}.
\end{Remark}

\section{Numerical Examples}
\label{SectionNumerics}
\subsection{The case of a diagonal frequency matrix}
Consider the Hamiltonian example introduced in \cite{LeBris:07}:
\begin{equation}
    \mathcal{H}=\frac{1}{2}p_x^2+\frac{1}{2}p_y^2+(x^2+y^2-1)^2+\frac{1}{2}(1+x^2)\omega^2 y^2
    \label{example}
\end{equation}

When $\omega=\epsilon^{-1/2} \gg 1$, bounded energy translates to initial conditions $x(0) \sim \omega y(0)$, which satisfy separation of timescales: $x$ is the slow variable, and $y$ is the fast. $K(x)=1+x^2$ is trivially diagonal. In addition to conservation of total energy, $I=\frac{p_y^2}{2 \sqrt{1+x^2}}+\frac{\sqrt{1+x^2} \omega^2 y^2}{2}$ is an adiabatic invariant.

\begin{figure}[ht]
\centering
\subfigure[The proposed method with coarse timestep $H=0.1$]{
\includegraphics[width=0.45\textwidth]{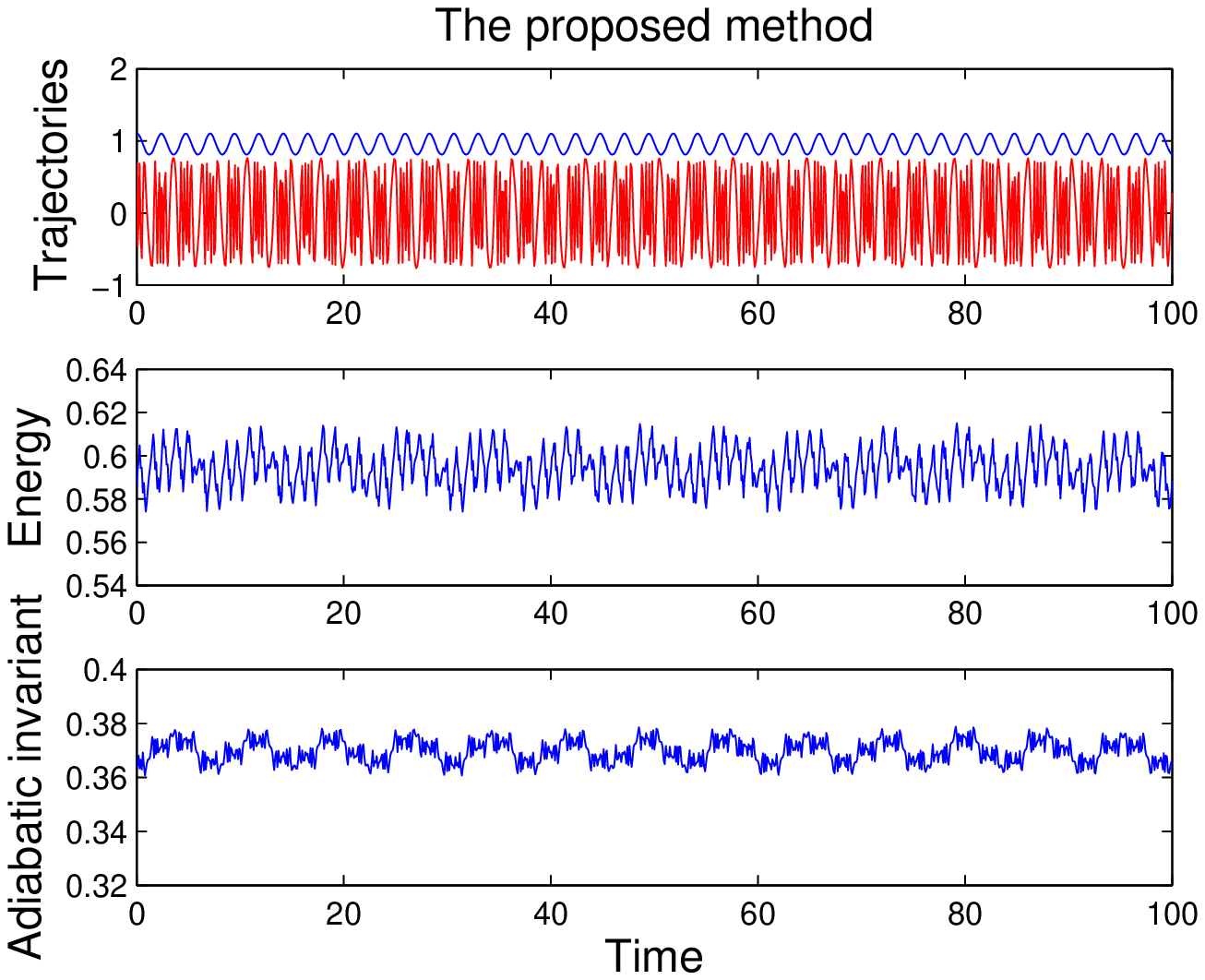}
\label{example1result1}
}
~
\subfigure[Variational Euler with small timestep $h=0.1/\omega=0.001$]{
\includegraphics[width=0.45\textwidth]{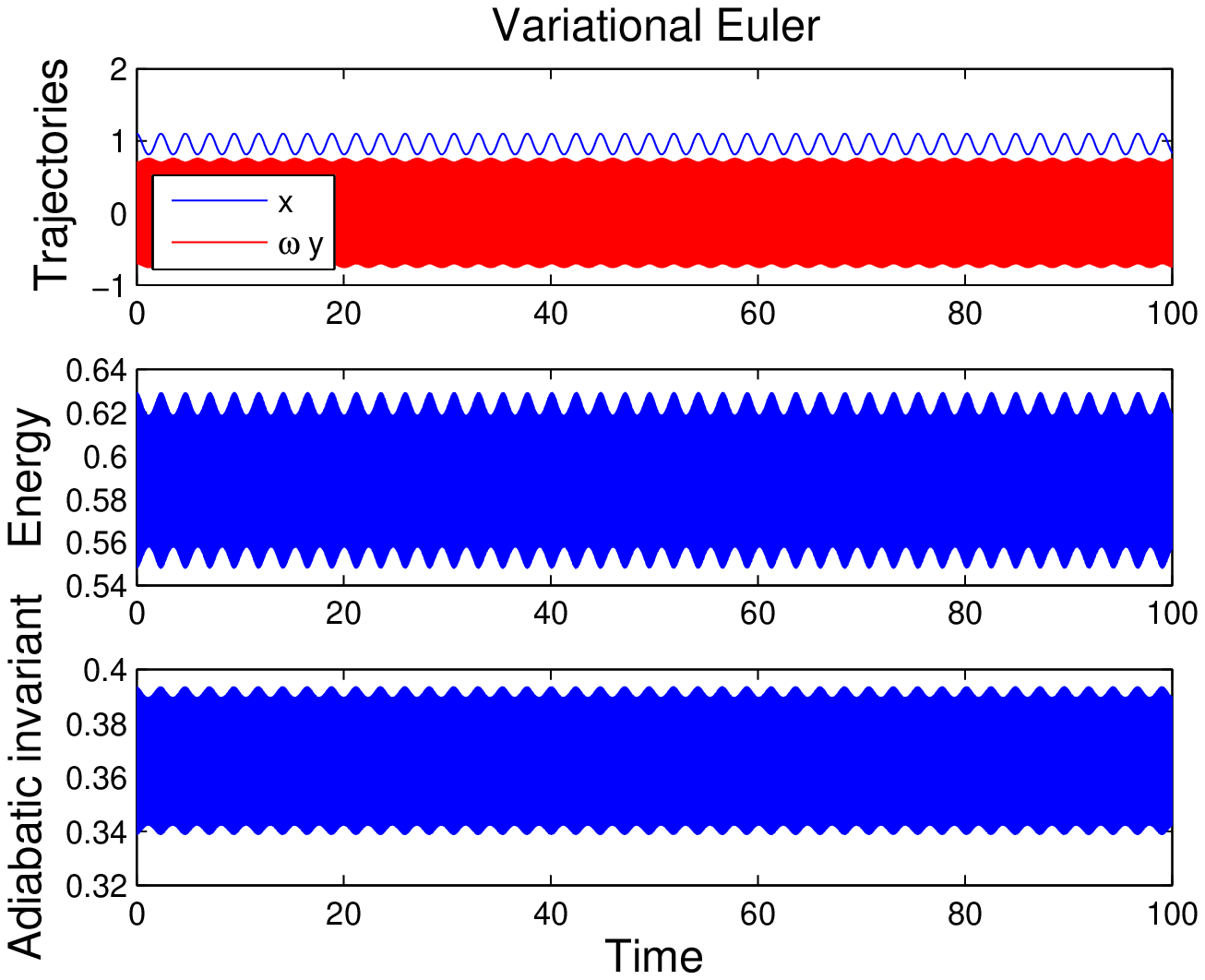}
\label{example1result2}
}
\subfigure[Very long time simulation by the proposed method with coarse timestep $H=0.1$]{
\includegraphics[width=0.45\textwidth]{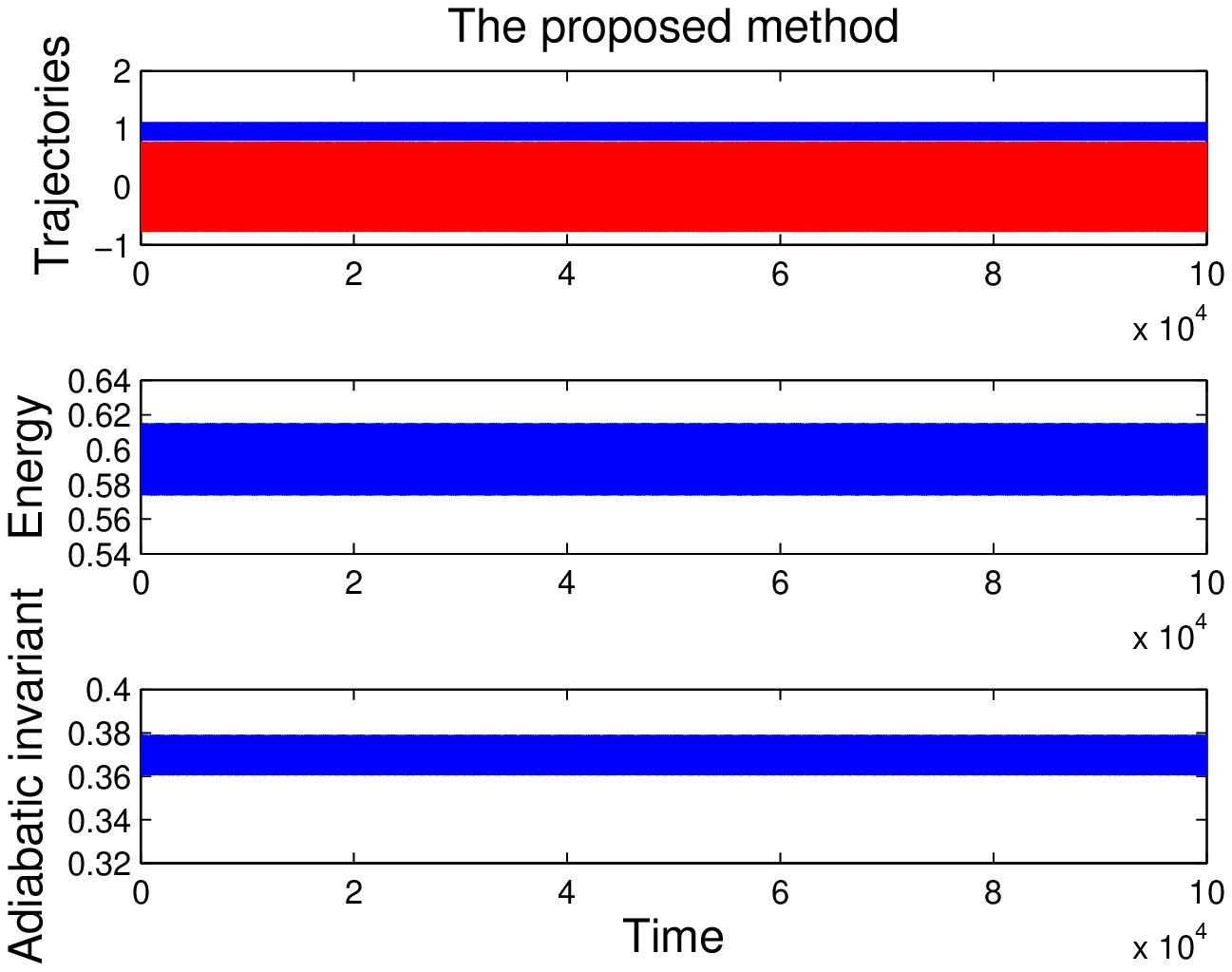}
\label{example1result3}
}
\caption{\footnotesize Simulations of a diagonal fast frequency example \eqref{example} by the proposed method and Variational Euler. $\omega=100$; $x(0)=1.1$, $y(0)=0.7/\omega$. }
\label{nonquad}
\end{figure}

\begin{figure}[ht]
\centering
\includegraphics[width=0.45\textwidth]{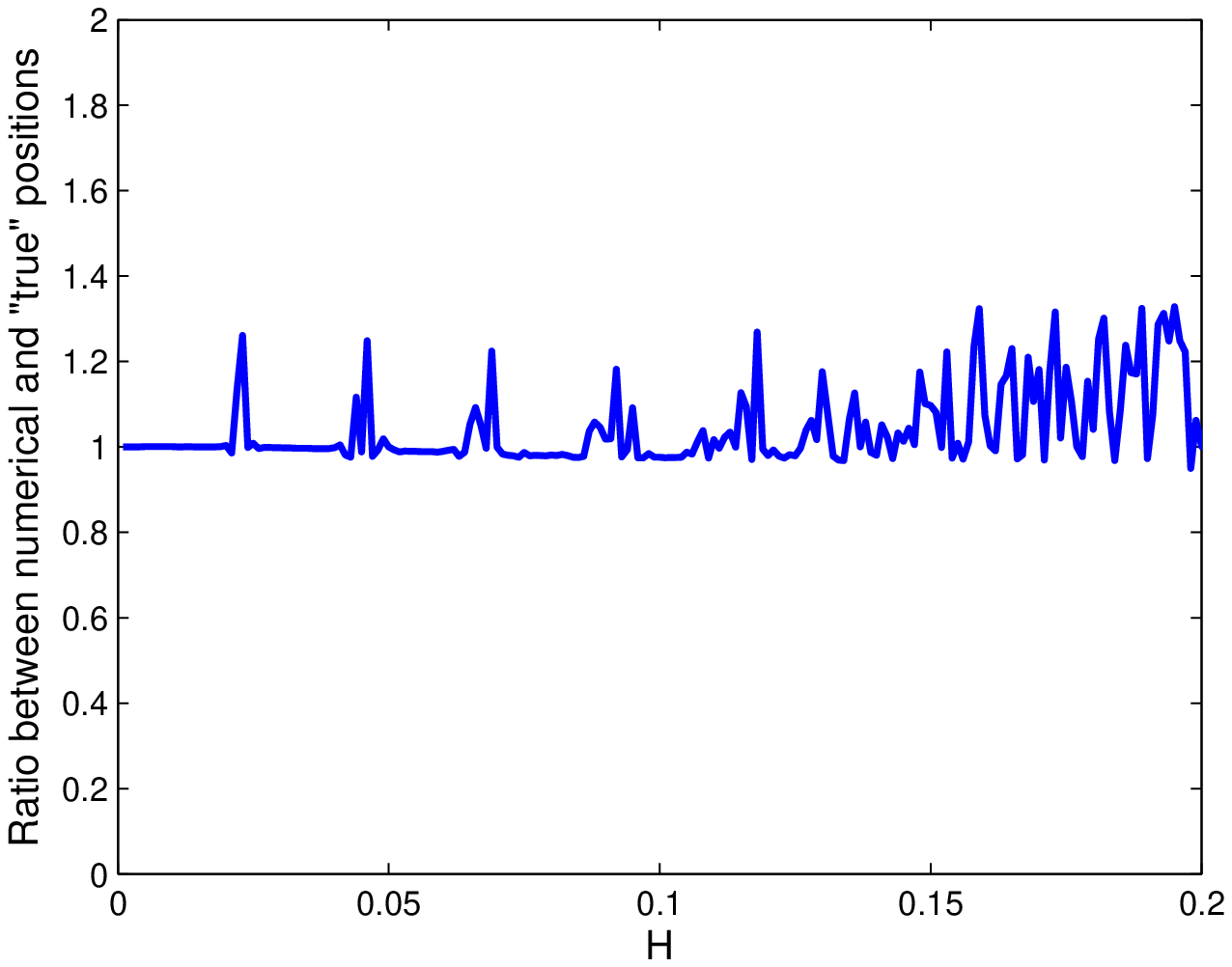}
\caption{\footnotesize Investigation on resonance frequencies of the proposed method on example \eqref{example}. The ratio between $x(T)|_{T=100}$ integrated by the proposed method integration and benchmark provides the ruler: a ratio closer to $1$ means a more accurate integration, and deviations from $1$ stand for step lengths that correspond to resonance frequencies. Time step $H$ samples from $0.001$ to $0.2$ with an increment of $0.001$. $\omega=100$; $x(0)=1.1$, $y(0)=0.7/\omega$. Benchmark is obtained by fine VE integration with $h=0.01/\omega$.}
\label{resonance}
\end{figure}

A comparison between Variational Euler and the proposed method is shown in Figure \ref{nonquad}.There it can be seen that preservations of energy and adiabatic invariant are numerically captured at least to a very large timescale. Since there is no overhead spent on matrix exponentiation here, an accurate $100$x speed up is achieved by the proposed method (because $H/h=100$).

It is known that the impulse method and its derivatives (such as mollified impulse methods) are not stable if the integration step falls in resonance intervals (mollified impulse methods have much narrower resonance intervals, which however still exist) \cite{Skeel:99, CalvoSena09}. Similarly, it will be very unnatural if the proposed method does not have resonance, because it reduces to a 1st-order version of impulse methods when there is no slow variable (Remark \ref{degenerateToImpulse}). In fact, in our numerical investigation (Figure \ref{resonance}), we clearly observe resonance frequencies before the integration step reaches the unstable limit (around $H\approx 0.5$), and widths of resonant intervals increase as $H$ grows for this particular example; however, we will not carry out a systematic analysis on resonance due to limitation of the length of a short communication.

\subsection{The case of a non-diagonal frequency matrix}
\label{nondiagonalExample}

Extend the previous example to a toy example of 3 degrees of freedom:
\begin{equation}
    \mathcal{H}=\frac{1}{2}p_x^2+\frac{1}{2}p_y^2+\frac{1}{2}p_z^2+(x^2+y^2+z^2-1)^2+\frac{1}{2} \omega^2 \begin{bmatrix} y \\ z \end{bmatrix}^T \begin{bmatrix} 1+x^2 & x^2-1 \\ x^2-1 & 3x^2 \end{bmatrix} \begin{bmatrix} y \\ z \end{bmatrix}
    \label{toyexample}
\end{equation}

It is easy to check that eigenvalues of $K(x)=\begin{bmatrix} 1+x^2 & x^2-1 \\ x^2-1 & 3x^2 \end{bmatrix}$ are both positive when $x>0.44$, which will always be true if the initial condition of $x$ stays close to $1$ and $\omega$ is big enough. In this case, bounded energy again implies $x(0) \sim \omega y(0) \sim \omega z(0)$ and gives well-separation of timescales: $x$ is the slow variable and $y$ and $z$ are the fast. $K(x)$ has its orthogonal frame for diagonalization as well as its eigenvalues slowly varying with time.

\begin{figure}[ht]
\centering
\subfigure[Till time 50]{
\includegraphics[width=0.60\textwidth] {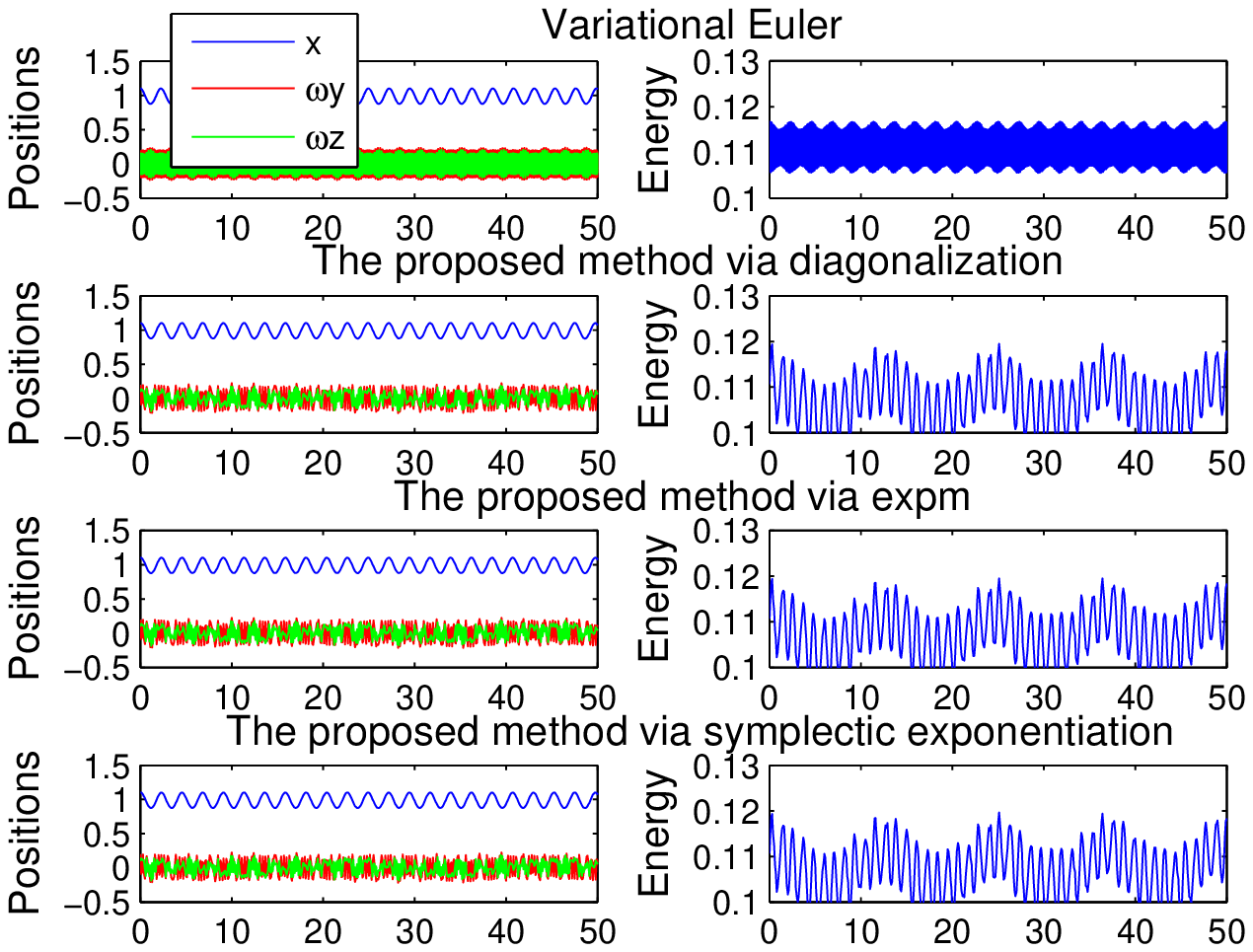}
\label{example2result1}
}
~
\subfigure[Till time 1000]{
\includegraphics[width=0.32\textwidth] {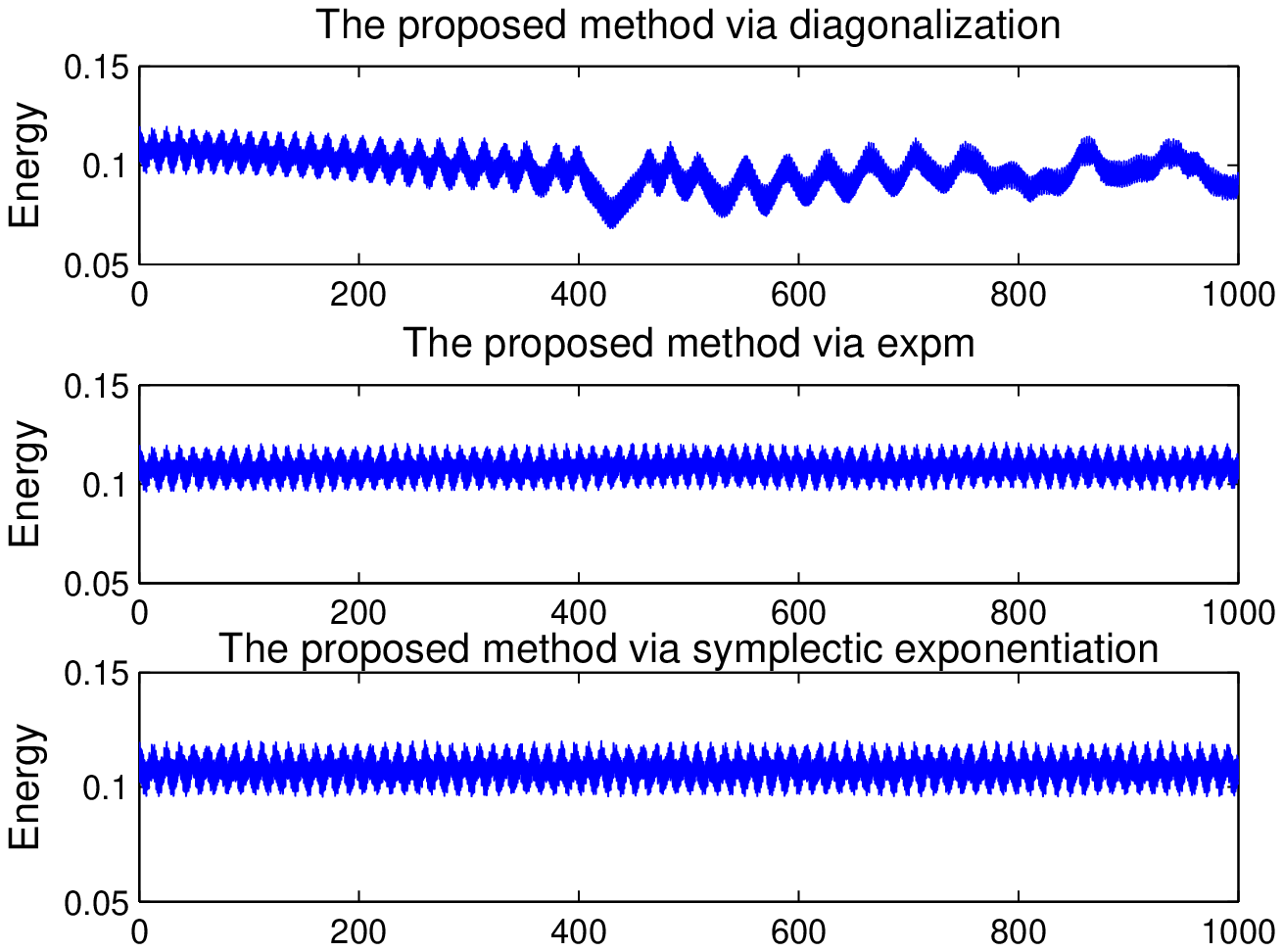}
\label{example2result2}
}
\caption{\footnotesize Simulations of a non-diagonal fast frequency example \eqref{toyexample} by Variational Euler, the proposed method with different implementations of matrix exponentiations. $\omega=100$, VE uses $h=0.1/\omega=0.001$ and the proposed method uses $H=0.1$ and $n=10$; $x(0)=1.1$, $y(0)=0.2/\omega$, $z(0)=0.1/\omega$, and initial momenta are zero.}
\label{nonquad_multiD}
\end{figure}

Figure \ref{nonquad_multiD} shows a comparison between Variational Euler, the proposed method with the matrix exponentiations computed by diagonalization and analytical integration (Eq. \ref{expViaDiagonalization}; diagonalization implemented by MATLAB command `diag'), and the proposed methods based on exponentiations (Eq. \ref{pureMatrixExponentiation} and \ref{phi3}) via MATLAB command `expm' \cite{ScalingSquaring} and via the fast matrix exponentiation method (Integrator \ref{fastMatrixExpScheme}). The default MATLAB matrix multiplication operation is used. All implementations of the proposed method are accurate, except that numerical errors in repetitive diagonalizations contaminated the symplecticity of the corresponding implementation over a long time simulation (as suggested by drifted energy), whereas other two implementations, respectively based on accurate but slow `expm' and fast symplectic exponentiations, do not have this issue. In a typical notebook run with MATLAB R2008b, the above four methods respectively spent 11.12, 0.23, 0.29 and 0.24 seconds on the same integration (till time 50), while 0, 0.14, 0.18, and 0.14 seconds were spent on matrix exponentiations. Computational gain by the symplectic exponentiation algorithm will be much more significant as the fast dimension becomes higher. Notice also that the computational gain by the proposed method over Variational Euler will go to infinity as $\epsilon\rightarrow 0$, even if the fast matrix exponentiation method is not employed.

\subsection{The case of a high-dimensional non-diagonal frequency matrix}
\label{nondiagonalHighExample}

Consider an arbitrarily high-dimensional example:
\begin{equation}
    \mathcal{H}=\frac{1}{2}p^2+\frac{1}{2}y^T y+(x^T x+q^2-1)^2+\frac{1}{2}\omega^2 x^T T(q) x
\end{equation}
where $q,p\in \mathbb{R}$ correspond to the slow variable, $x,y\in \mathbb{R}^{d_f}$ correspond to fast variables, and $T(q)$ is the following Toeplitz matrix valued function:
\begin{equation}
    T(q)=\begin{bmatrix} 1 & \hat{q}^1 & \hat{q}^2 & \ldots & \hat{q}^{d_f-1} \\
                         \hat{q}^1 & 1 & \hat{q}^1 & \ldots & \hat{q}^{d_f-2} \\
                         \hat{q}^2 & \hat{q}^1 & 1 & \ldots & \hat{q}^{d_f-3} \\
                         & & \vdots & & \\
                         \hat{q}^{d_f-1} & \hat{q}^{d_f-2} & \hat{q}^{d_f-3} & \ldots & 1 \end{bmatrix}
    \label{highDexample}
\end{equation}
where $\hat{q}=q/2$ so that eigenvectors and eigenvalues vary slowly with $q$ given an initial condition of $q(0)\approx 1$. Note that the expression of $T(\cdot)$ is highly nonlinear.

\begin{figure}[ht]
\includegraphics[width=\textwidth] {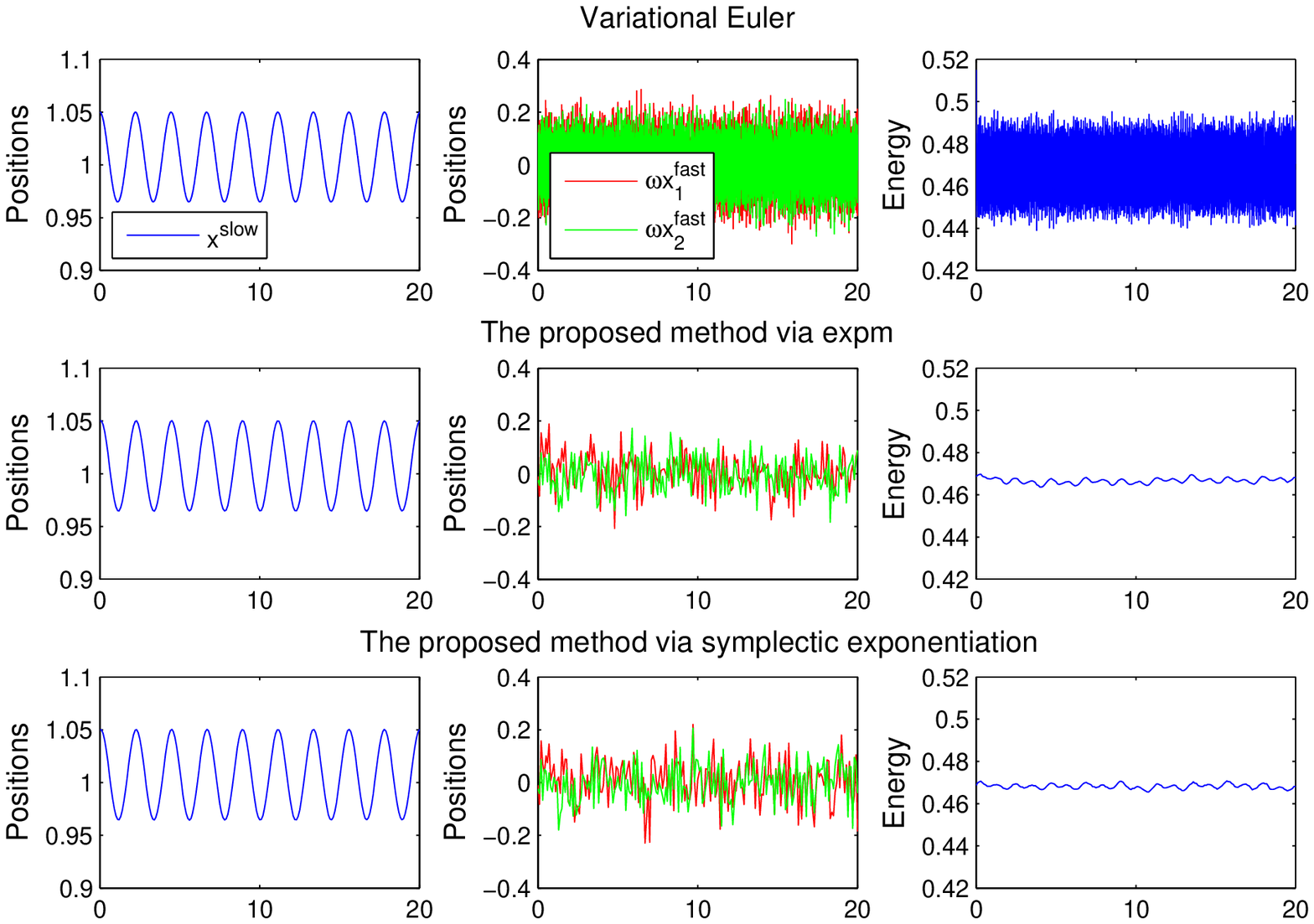}
\caption{\footnotesize Simulations of a non-diagonal fast frequency high-dimensional example \eqref{highDexample} by Variational Euler, the proposed method via MATLAB matrix exponentiation `expm,' and the proposed method via fast matrix exponentiations ($n=10$). Fast variable dimensionality is $d_f=100$. $\omega=1000$. VE uses $h=0.1/\omega$ and the proposed method uses $H=0.1$, $q(0)=1.05$, $x(0)$ is a $d_f+1$-dimensional vector with independent and identically distributed components that are normal random variables with zero mean and variance of $1/\omega/\sqrt{d_f}$ (so that energy is bounded), and initial momenta are zero. Only trajectories of the first two fast variables were drawn for clarity.}
\label{nonquad_highD}
\end{figure}

We present in Figure \ref{nonquad_highD} a comparison between Variational Euler and the proposed methods with the matrix exponentials computed by MATLAB command `expm' and by the fast matrix exponentiation method (Integrator \ref{fastMatrixExpScheme}) on a high dimensional example with $d_f=100$. Accuracy-wise, the proposed method simulations yield results similar to VE (note that fast variables are not fully resolved due to a coarse time step that is larger than their periods). Speed-wise, Variational Euler, the proposed methods via `expm' and via symplectic exponentiation respectively spent 136.7, 66.0 and 12.0 seconds on the same integration, while 65.7 and 11.7 seconds were spent on matrix exponentiation operations in the latter two. Notice that if Coppersmith-Winograd \cite{CoWi90} is used to replace MATLAB matrix multiplication, the number 11.7 should be further reduced. In spite of that, the proposed method with the proposed matrix exponentiation scheme already holds a dominant speed advantage, and this advantage will be even more significant if $\omega$ and/or $d_f$ is further increased (results not shown).

\section{Related work}
\paragraph{Stiff integration:}
  Many elegant methods have been proposed in the area of stiff Hamiltonian integration, and some are closely related to this work. An incomplete list will be discussed here.

Impulse methods \cite{Grubmuller:91,Tuckerman:92} admit uniform error bounds on positions and can be categorized as splitting methods \cite{SIM1}.
In their abstract form, impulse methods are not limited to quadratic stiff potentials; however, their practical implementation requires
an approximation of the flow associated with the stiff potential.
Our method is based on a generalization of the impulse method to (possibly high-dimensional) situations where the stiff potential contains a slowly varying component. Although simple in its abstract expression, the practical implementation of this generalization (for high-dimensional systems) has required the introduction of a non-trivial symplectic matrix exponentiations scheme.

Impulse methods have been mollified \cite{Skeel:99,Sanz-Serna:08} to gain extra stability and accuracy. However, mollified impulse methods and other members of the exponential integrator family \cite{ExponentialIntegrators}, for instance Gautschi-type integrators \cite{Gautschi}, are not based on splitting, and hence the splitting approach in this paper does not immediately generalize them.

The reversible averaging integrator proposed in \cite{LeRe01} averages the force on slow variables and avoids resonant instabilities. It treats the dynamics of slow and fast variables separately and assumes piecewise linear trajectories of the slow variables, both in the same spirit as in our proposed method; it is, however, not symplectic, although reversible.

Implicit methods, for example LIN \cite{ZhSc:93}, work for generic stiff Hamiltonian systems, but \emph{implicit methods in general fail to capture the effective dynamics of the slow time scale because they cannot correctly capture non-Dirac invariant distributions} \cite{LiAbE:08}, and they are generally slower than explicit methods if comparable step lengths are employed.

IMEX is a variational integrator for stiff Hamiltonian systems \cite{Stern:09}. It works by introducing a discrete Lagrangian via trapezoidal approximation of the soft potential and midpoint approximation of the stiff potential. It is explicit in the case of quadratic fast potential, but is implicit in the case of quasi-quadratic fast potentials.

A Hamilton-Jacobi approach is used to derive a homogenization method for multiscale Hamiltonian systems \cite{LeBris:07}, which works for quasi-quadratic fast potentials with scalar frequency and yields a symplectic method. We also refer to \cite{DoLeLe10} for a generalization of this method to systems that have either one varying fast frequency or several constant frequencies. The difficulty with this elegant analytical approach would be to deal with high-dimensional systems.

Other generic multiscale methods that integrate the slow dynamics by averaging the effective contribution of the fast dynamics include: Heterogeneous Multiscale Methods (HMM) \cite{MR2164093, MR2314852, CaSer08, Ariel:08, Seamless09}, the equation-free method \cite{MR2041455, GiKeKup06, KevGio09}, and FLow AVeraging integratORS (FLAVORS) \cite{FLAVOR09}.  Those methods can be applied to a much broader spectrum of problems than considered here. However, they all essentially use a mesoscopic timestep, which is usually one or two orders of magnitude smaller than the coarse step employed here. Moreover, symplecticity is a big concern. In their original form, Heterogeneous Multiscale Methods and equation-free methods are based on the averaging of the instantaneous drifts of slow variables, which breaks symplecticity in all variables. Reversible and symmetric HMM generalizations have been proposed \cite{Ariel:09, SerArTs09}. FLAVORS \cite{FLAVOR09} are based on averaging instantaneous flows by turning on and off stiff coefficients in legacy integrators used as black boxes. In particular, they do not require the identification of slow variables and inherit the symplecticity and reversibility of the legacy integrators that they are derived from.

\paragraph{Matrix exponentiation:} In the case of quasi-quadratic stiff potentials, the proposed algorithm exponentiates a slowly varying matrix at each time step. When the elasticity matrix $K$ is not diagonalizable by a constant orthogonal transformation, a numerical algebra algorithm is employed for that calculation at the expense of $\mathcal{O}(n)$ operations of $d_f$-by-$d_f$ matrix multiplications per time step, where $d_f$ is the dimension of fast variable (and hence $K$), and $n$ is a preset constant that is at most $\log(\epsilon^{-1})$.


There are various approaches to exponentiate a matrix, including diagonalization, series methods, scaling and squaring, ODE solving, polynomial methods, matrix decomposition methods, and splitting, etc., as comprehensively reviewed in \cite{MoLo:03}. Many of these methods, however, differ from our approach here in that they do not guarantee that the resulting implementation of the proposed method to be symplectic as it analytically should be, unless high precision (hence slow computation) is required; most of them could not even guarantee a symplectic approximation to $F_2$ and $F_3$.

The proposed approach (Integrator \ref{fastMatrixExpScheme}) obtains its efficiency by a trick of self-multiplication, which was previously used in the method of scaling and squaring \cite{ScalingSquaring}. However, the Pad\'{e} approximation used in scaling and squaring is replaced by a symplectic and reversible approximation based on the Verlet integrator. Consequently, symplecticity and better efficiency are obtained, and accuracy is kept. Improvements by this numerical exponentiation over `expm' and `diag,' in terms of both accuracy and speed, are observed numerically in Subsection \ref{nondiagonalExample} and \ref{nondiagonalHighExample}.

For our alternative approach (see \eqref{fastMatrixExp} for the general strategy and Appendix for implementational details for the specific purpose of multiscale integration), it uses the slowly varying property of the matrix to repetitively modify the exponential from the previous step by a small symplectic change to get a new exponential. Regarding updating matrix exponentials, since there are results such as \cite{EigenPerturbation} on relationships between perturbed eigenvalues and perturbation in the matrix, a natural thought is to use eigenstructures that were explored in the previous step as initial conditions in iterative algorithms (such as Jacobi-Davidson for eigenvalues \cite{Jacobi-Davidson} or Rayleigh Quotient for extreme eigenvalues \cite{NumericalLinearAlgebra}). This idea, however, did not significantly accelerate the computation as we explored in numerical experiments with an incomplete pool of methods. Other matrix decompositions methods (QR for instance) did not gain much from previous decompositions either in our numerical investigations. Our way of exponential updating is essentially an operator splitting approach, which is analogous to our main vector field splitting strategy that yields the proposed multiscale integrator.

\section*{Acknowledgement}
This work was supported by the National Science Foundation [CMMI-092600].
We sincerely thank Charles Van Loan for a stimulating discussion and Sydney Garstang for proofreading the manuscript.
We are also grateful to two anonymous referees for precise and detailed  comments and suggestions.

\section*{Appendix: an alternative matrix exponentiation scheme}
We will present in Integrator \ref{fastMatrixExpSchemeAlternative} an alternative (symplectic) way for computing $F_{3,k}$ and $G_{2,k,i}$. This alternative is based on iteratively updating the matrix exponential from the computation at the previous step. We will first demonstrate its full version, and then provide a simple approximation which is not exactly symplectic on all variables but symplectic on the fast variables (in the sense of a symplectic submanifold) and exhibits satisfactory long time performance in numerical experiments.

\begin{lemma}
    Define
    \begin{equation}
    \begin{bmatrix}
        \alpha(t) & \beta(t) & \gamma(t) \\
        0 & F_2(t) & G_2(t) \\
        0 & 0 & F_3(t)
    \end{bmatrix}:=\exp\left(
    \begin{bmatrix}
        -N^T & M J & 0 \\
        0 & -N^T & M \\
        0 & 0 & N
    \end{bmatrix} t \right)
    \label{llqgr3hubqiqbr}
    \end{equation}
    Then for any $H$, we have $-F_3(H)^T \gamma(H)=\int_0^H F_3^T(s) M (-JG_2(t)) \, ds$.
    \label{2ndDerivativePre}
\end{lemma}
\begin{proof}
    Differentiating \eqref{llqgr3hubqiqbr} with respect to $t$ and equating each matrix component on left and right hand sides, we obtain:
    \begin{equation}
    \begin{cases}
        \dot{\alpha}=-N^T \alpha \\
        \dot{F_2}=-N^T F_2 \\
        \dot{F_3}=N F_3 \\
        \dot{\beta}=-N^T \beta + M J F_2 \\
        \dot{G_2}=-N^T G_2 + M F_3 \\
        \dot{\gamma}=-N^T \gamma + M J G_2
    \end{cases}
    \end{equation}
    where the initial conditions obviously are $\alpha(0)=I, F_2(0)=I, F_3(0)=I, \beta(0)=0, G_2(0)=0, \gamma(0)=0$.

    Solving these inhomogeneous linear equations leads to known results including $F_2(t)=\exp(-N^T t)$, $F_3(t)=\exp(Nt)$ and $G_2(t)=\int_0^t \exp(-N^T(t-s)) M \exp(N s)\, ds$, as well as new results such as
    \begin{equation}
        \gamma(t)=\int_0^t \exp(-N^T(t-s)) M J G_2(s) \, ds ,
    \end{equation}
    which is equivalent to
    \begin{equation}
        -F_3(H)^T \gamma(H)=\int_0^H F_3(s)^T M (-J G_2(s)) \, ds
    \end{equation}
\end{proof}

\begin{lemma}
    If $M=M^T$, $F_2^T F_3=I$ and $\partial F_3=-J G_2$, such as those derived from $N$ and $M$ defined in Section \ref{gguqirehugqubliafd}, then
    \begin{equation}
        \partial G_2(H)=F_2(H)\left(-(F_3(H)^T \gamma(H))^T-F_3(H)^T \gamma(H) + \int_0^H F_3(s)^T \partial M F_3(s) \, ds - (-J G_2(H))^T G_2(H)\right)
    \end{equation}
    \label{2ndDerivative}
\end{lemma}
\begin{proof}
    By Leibniz's rule
    \begin{equation}
        \partial G_2(H)=[F_3(H)^T]^{-1}\left(\partial \left(F_3(H)^T G_2(H)\right) - \partial F_3(H)^T G_2(H)\right)
    \end{equation}
    By the definition of $F_3$ and $G_2$, this is
    \begin{equation}
        \partial G_2(H)=F_2(H)\left(\partial \left(\int_0^H F_3(s)^T M F_3(s) \, ds\right) - (-J G_2(H))^T G_2(H)\right) ,
    \end{equation}
    in which
    \begin{align}
        \partial \left(\int_0^H F_3(s)^T M F_3(s) \, ds\right) &= \int_0^H \partial F_3(s)^T M F_3(s) \, ds + \int_0^H F_3(s)^T M \partial F_3(s) \, ds + \int_0^H F_3(s)^T \partial M F_3(s) \, ds \nonumber\\
        &= \int_0^H (-J G_2(s))^T M F_3(s) \, ds + \int_0^H F_3(s)^T M (-J G_2(s)) \, ds + \int_0^H F_3(s)^T \partial M F_3(s) \, ds \nonumber\\
        &= -(F_3(H)^T \gamma(H))^T-F_3(H)^T \gamma(H) + \int_0^H F_3(s)^T \partial M F_3(s) \, ds
    \end{align}
    for $\gamma(H)$ defined in Lemma \ref{2ndDerivativePre}.
\end{proof}

\begin{Remark}
    $\int_0^H F_3(s)^T \partial M F_3(s) \, ds=\tilde{F}_3(H)^T \tilde{G}_2(H)$ can be computed by again using the trick of:
    \begin{equation}
    \begin{bmatrix}
        \tilde{F}_2(t) & \tilde{G}_2(t) \\
        0 & \tilde{F}_3(t)
    \end{bmatrix}:=\exp\left(
    \begin{bmatrix}
        -N^T & \partial M \\
        0 & N
    \end{bmatrix} t \right)
    \end{equation}
    \label{2ndDerivativeRemark}
    Of course, to get $B'_{0,i,j}=\partial_j B_{0,i}$, we use the fact that $B_{0,i}=G_{2,0,i}(H/2^n)$.
\end{Remark}

\begin{lemma}
    Suppose $q^{fast}_{(k+1)'},p^{fast}_{(k+1)'},q^{slow}_{(k+1)'},p^{slow}_{(k+1)'}$ are obtained from $q^{fast}_{k'},p^{fast}_{k'},q^{slow}_{k'},p^{slow}_{k'}$ by Integrator \ref{extended_SIM} with $F_{3,k}$ and $G_{2,k,i}$ satisfying $F_{3,k}^T J F_{3,k}=J$ and $G_{2,k,i}=-J \frac{\partial}{\partial q^{slow}_{k',i}} F_{3,k}$, then
    \begin{equation}
        \frac{\partial q^{slow}_{(k+1)',i}}{\partial q^{slow}_{k',j}}=I+\frac{H}{2} \begin{bmatrix} q^{fast}_{k'} \\ p^{fast}_{k'} \end{bmatrix} \left( G_{2,k,j}(H)^T J G_{2,k,i}(H) + F_{3,k}(H)^T \frac{\partial}{\partial q^{slow}_{k'}} G_{2,k}(H) \right) \begin{bmatrix} q^{fast}_{k'} \\ p^{fast}_{k'} \end{bmatrix}
        \label{uqiarebgpgot14t13t}
    \end{equation}
    \label{LemmaChainRule}
\end{lemma}
\begin{proof}
    Using chain rule, we have:
    \begin{equation}
        \frac{\partial q^{slow}_{(k+1)',i}}{\partial q^{slow}_{k',j}}=I+\frac{H}{2} \begin{bmatrix} q^{fast}_{k'} \\ p^{fast}_{k'} \end{bmatrix} \left( \frac{\partial}{\partial q^{slow}_{k',j}} F_{3,k}(H)^T G_{2,k,i}(H) + F_{3,k}(H)^T \frac{\partial}{\partial q^{slow}_{k',j}} G_{2,k,i}(H) \right) \begin{bmatrix} q^{fast}_{k'} \\ p^{fast}_{k'} \end{bmatrix}
    \end{equation}
    This simplifies to \eqref{uqiarebgpgot14t13t} because $G_{2,k,i}=-J \frac{\partial}{\partial q^{slow}_{k',i}} F_{3,k}$ and $-J^T=J$.
\end{proof}

\begin{Integrator}
Iterative matrix exponentiation scheme (alternative to Integrator \ref{fastMatrixExpScheme}) that obtains $F_{3,k}$ and $G_{2,k,i}$ via symplectic updates. $k$ is the same index as the one used in Integrator \ref{extended_SIM}. $n\geq 1$ is an integer controlling the accuracy of matrix exponential approximations.

\begin{enumerate}
\item
    At the beginning of simulation, let $q^{slow}_{0'}=q^{slow}_0+H p^{slow}_0$ and evaluate $K_0:=K(q^{slow}_{0'})$ and $\partial_i K_0:=\frac{\partial}{\partial q^{slow}_{0',i}} K(q^{slow}_{0'})$ ($i=1,\ldots,d_s$). Calculate $\begin{bmatrix} A_0 & B_{0,i} \\ 0 & C_0 \end{bmatrix}:=\exp \left( \begin{bmatrix} -N_0^T & M_{0,i} \\ 0 & N_0 \end{bmatrix} H/2^n \right)$ by any favorite matrix exponentiation method (e.g., by the symplectic method introduced in Section \ref{gguqirehugqubliafd}), where $N_0=\begin{bmatrix} 0 & I \\ -\epsilon^{-1}K_0 & 0 \end{bmatrix}$ and $M_{0,i}=\begin{bmatrix} \epsilon^{-1}\partial_i K_0 & 0 \\ 0 & 0 \end{bmatrix}$.
\item
    Compute $B'_{0,i,j}= \frac{\partial}{\partial q^{slow}_{0',j}} B_{0,i}$. One cheap way to do so is to use Lemma \ref{2ndDerivative} with Remark \ref{2ndDerivativeRemark}.
\item
    Start the updating loop, with the step count indicated by $k$ starting from 1; let $q^{slow,fast}_1=q^{slow,fast}_0$, $p^{slow,fast}_1=p^{slow,fast}_0$, and $\frac{q^{slow}_{1'}}{q^{slow}_{0'}}=I$;
\item
    Carry out the $q_k,p_k \mapsto q_{k'},p_{k'}$ half-step (in Integrator \ref{extended_SIM}). Evaluate $K_k:=K(q^{slow}_{k'})$, and let $D_k:=\begin{bmatrix} 0 & \epsilon^{-1} (K_k^T-K_{k-1}^T) H/2^n \\ 0 & 0 \end{bmatrix}$. Define $A_{k}:=A_{k-1} \exp(D_k)$ and use the equality $\exp(D_k)=I+D_k$ (since $D_k$ is nilpotent); similarly, define $C_k:=C_{k-1} \exp(-D_k^T)=C_{k-1}-C_{k-1}D_k^T$;
    \label{loopStarted}
\item
    Let $B_{k,i}=-J \frac{\partial C_k}{\partial q^{slow}_{k',i}}$, which can be computed from known values using chain rule:
    \begin{align}
        B_{k,i} &= -J \frac{\partial (C_{k-1}(I+D_k))}{\partial q^{slow}_{k',i}}
        =-J\left(\sum_{j=1}^{d_s} \frac{\partial q^{slow}_{(k-1)',j}}{\partial q^{slow}_{k',i}} \frac{\partial C_{k-1}}{\partial q^{slow}_{{k-1}',j}} (I+D_k)+C_{k-1} \frac{\partial D_k}{\partial q^{slow}_{k',i}}\right) \nonumber\\
        &= \sum_{j=1}^{d_s} \frac{\partial q^{slow}_{(k-1)',j}}{\partial q^{slow}_{k',i}} B_{k-1,j}(I+D_k)+C_{k-1} \frac{\partial D_k}{\partial q^{slow}_{k',i}}
    \end{align}
    To compute $\frac{\partial D_k}{\partial q^{slow}_{k',i}}$, we need the derivatives of $K_k$ and $K_{k-1}$ with respect to $q^{slow}_{k',i}$; the former is trivial, and the latter again can be computed by chain rule:
    \begin{equation}
        \frac{\partial K_{k-1}^T}{\partial q^{slow}_{k',i}}=\sum_{j=1}^{d_s} \frac{\partial q^{slow}_{(k-1)',j}}{\partial q^{slow}_{k',i}} \frac{\partial K_{k-1}^T}{\partial q^{slow}_{{k-1}',j}}
    \end{equation}
    \label{updateBasDerivative}
\item
    $B'_{k,i,j}$ can be similarly computed from $B'_{k-1,i,j}$, $B_{k-1,i}$, $C_{k-1}$ and $D_k$ by repetitively applying chain rule. The detail is lengthy and hence omitted.
\item  Let $F_{2,k}^1:=A_k$, $G_{2,k,i}^1:=B_{k,i}$, $F_{3,k}^1:=C_k$, then repetitively apply $\begin{bmatrix} F_{2,k}^{j+1} & G_{2,k,i}^{j+1} \\ 0 & F_{3,k}^{j+1} \end{bmatrix}: = \begin{bmatrix} F_{2,k}^j & G_{2,k,i}^j \\ 0 & F_{3,k}^j \end{bmatrix}^2 = \begin{bmatrix} F_{2,k}^j F_{2,k}^j & F_{2,k}^j G_{2,k,i}^j + G_{2,k,i}^j F_{3,k}^j \\ 0 & F_{3,k}^j F_{3,k}^j \end{bmatrix}$ for $j=1,\ldots,n$, and finally define $F_{2,k}:=F_{2,k}^{n+1}$, $G_{2,k,i}:=G_{2,k,i}^{n+1}$, $F_{3,k}=F_{3,k}^{n+1}$.
    \label{exponential_repeat2}
\item
    Compute $\frac{\partial q^{slow}_{(k+1)',i}}{\partial q^{slow}_{k',j}}$ by using Lemma \ref{LemmaChainRule}, so that it could be used by Step \ref{updateBasDerivative} for the next $k$. $\frac{\partial}{\partial q^{slow}_{k',j}}G_{2,k,i}(H)$ is computed based on the following:
    \begin{equation}
        \frac{\partial}{\partial q^{slow}_{k',j}} (A_k B_{k,i}+B_{k,i} C_k) = -A_k B_{k,j}^T J A_k B_{k,i} + A_k B'_{k,i,j} + B'_{k,i,j} C_k - B_{k,i} J B_{k,j}
    \end{equation}
    where the first term is due to $\frac{\partial A_k}{\partial q^{slow}_{k',j}}=-A_k B_{k,j}^T J A_k$, which is because $\partial A^T C + A^T \partial C=\partial (A^T C)=\partial I=0$ and therefore $\partial A^T = -A^T \partial C C^{-1} = A^T J B C^{-1}=A^T J B A^T$.
    A similar trick of self multiplication applies to get the derivative of the $2^n$-times product.
\item
    Carry out the $q_{k'},p_{k'} \mapsto q_{k+1},p_{k+1}$ half-step update of numerical integration using $F_{2,k}$, $F_{3,k}$ and $G_{2,k,i}$, and then increase $k$ by 1 and go to Step \ref{loopStarted} until integration time is reached.
\end{enumerate}
\label{fastMatrixExpSchemeAlternative}
\end{Integrator}

$F_{3,k}$ and $G_{2,k,i}$ computed in this way (Integrator \ref{fastMatrixExpSchemeAlternative}) will also satisfy Lemma \ref{lemma_flowSymplecticity} and render the integration symplectic on all variables. Proofs are omitted but they are analogous to those in Section \ref{SectionAnalysisSymplecticity}, and all structures, such as reversibility, symplecticity of $F_2$ and $F_3$ (illustrated by corresponding lemmas), and the relation between $F_3$ and $G_2$, will be preserved as long as they are satisfied by $A_0,B_{0,i},C_0$ (i.e., the initial matrix exponentiation is accurate).

In terms of efficiency, this method only uses one single matrix exponentiation operation and then keeps on updating it. Nevertheless, it is not easy to implement, and its speed advantage is not dominant. However, if the requirement on symplecticity is not that strict and a small numerical error in the matrix exponential is allowed (recall an analogous case of the famous implicit mid-point integrator, in which implicit solves are in fact not done perfectly and continuously polluting the symplecticity), we could use the approximation of $\frac{\partial q^{slow}_{(k+1)',i}}{\partial q^{slow}_{k',j}}=I$. This will introduce a local error of $\mathcal{O}(Hn/2^n)$ in $G_{2,k,i}$ at each timestep (details omitted), but the local error in $F_{2,k}$ and $F_{3,k}$ is 0, and the method is symplectic on the submanifold of the fast variables (although not symplectic on all variables). The approximating method is:

\begin{Integrator}
An efficient approximation of Integrator \ref{fastMatrixExpSchemeAlternative}:
\begin{enumerate}
\item
    At the beginning of simulation, let $q^{slow}_{0'}=q^{slow}_0+H p^{slow}_0$ and evaluate $K_0:=K(q^{slow}_{0'})$ and $\partial_i K_0:=\partial_i K(q^{slow}_{0'})$ ($i=1,\ldots,d_s$) and calculate $\begin{bmatrix} A_0 & B_{0,i} \\ 0 & C_0 \end{bmatrix}:=\exp \left( \begin{bmatrix} -N_0^T & M_{0,i} \\ 0 & N_0 \end{bmatrix} H/2^n \right)$ by any favorite matrix exponentiation method, where $N_0=\begin{bmatrix} 0 & I \\ -\epsilon^{-1}K_0 & 0 \end{bmatrix}$ and $M_{0,i}=\begin{bmatrix} \epsilon^{-1}\partial_i K_0 & 0 \\ 0 & 0 \end{bmatrix}$; let $q^{slow,fast}_1=q^{slow,fast}_0$ and $p^{slow,fast}_1=p^{slow,fast}_0$.
\item
    Start the updating loop, with the step count indicated by $k$ starting from 1;
\item
    Carry out the $q_k,p_k \mapsto q_{k'},p_{k'}$ half-step. Evaluate $K_k:=K(q^{slow}_{k'})$ and $\partial_i K_k:=\partial_i K(q^{slow}_{k'})$, let $D_k:=\begin{bmatrix} 0 & \epsilon^{-1} (K_k^T-K_{k-1}^T) H/2^n \\ 0 & 0 \end{bmatrix}$ and $E_{k,i}:=\begin{bmatrix} \epsilon^{-1} (\partial_i K_k-\partial_i K_{k-1}) H/2^n & 0 \\ 0 & 0 \end{bmatrix}$. Define $\begin{bmatrix} A_{k} & B_{k,i} \\ 0 & C_{k} \end{bmatrix}:=\begin{bmatrix} A_{k-1} & B_{k-1,i} \\ 0 & C_{k-1} \end{bmatrix} \times \exp \begin{bmatrix} D_k & E_{k,i} \\ 0 & -D_k^T \end{bmatrix}$ and use the equality $\exp \begin{bmatrix} D_k & E_{k,i} \\ 0 & -D_k^T \end{bmatrix} =\begin{bmatrix} I+D_k & E_{k,i} \\ 0 & I-D_k^T \end{bmatrix}$ (because $D_k E_{k,i}=0$ and $E_{k,i} D_k^T=0$) to evaluate $A_k=A_{k-1}+A_{k-1}D_k$, $B_{k,i}=B_{k-1,i}+A_{k-1}E_{k,i}-B_{k-1,i}D_k^T$, and $C_k=C_{k-1}-C_{k-1}D_k^T$;
    \label{loopStarted2}
\item  Let $F_{2,k}^1:=A_k$, $G_{2,k,i}^1:=B_{k,i}$, $F_{3,k}^1:=C_k$, then repetitively apply $\begin{bmatrix} F_{2,k}^{j+1} & G_{2,k,i}^{j+1} \\ 0 & F_{3,k}^{j+1} \end{bmatrix}: = \begin{bmatrix} F_{2,k}^j & G_{2,k,i}^j \\ 0 & F_{3,k}^j \end{bmatrix}^2 = \begin{bmatrix} F_{2,k}^j F_{2,k}^j & F_{2,k}^j G_{2,k,i}^j + G_{2,k,i}^j F_{3,k}^j \\ 0 & F_{3,k}^j F_{3,k}^j \end{bmatrix}$ for $j=1,\ldots,n$, and finally define $F_{2,k}:=F_{2,k}^{n+1}$, $G_{2,k,i}:=G_{2,k,i}^{n+1}$, $F_{3,k}=F_{3,k}^{n+1}$.
    \label{exponential_repeat3}
\item
    Carry out the $q_{k'},p_{k'} \mapsto q_{k+1},p_{k+1}$ half-step update of numerical integration using $F_{2,k}$, $F_{3,k}$ and $G_{2,k,i}$, and then increase $k$ by 1 and go to Step \ref{loopStarted2} until integration time is reached.
\end{enumerate}
\end{Integrator}

Numerical experiments presented in Section \ref{SectionNumerics} are repeated using this approximating integrator. Energy preservations are as good as before, and slow trajectories show no significant deviation, suggesting no significant effect of the approximated symplecticity (detailed results omitted). This approximation, on the other hand, allows a choice of an even smaller $n$, such as $n=5$ for the previous examples, which results in a further speed-up.

\def\cprime{$'$} \def\cprime{$'$} \def\cprime{$'$}
  \def\cydot{\leavevmode\raise.4ex\hbox{.}}

\end{document}